\documentclass[10pt,oneside]{article}
\usepackage{amsmath, amsfonts, amsthm, amssymb, amscd}
\usepackage[mathcal]{euscript} 
\usepackage{dsfont, pxfonts, verbatim}
\usepackage{mathrsfs}  
\newtheorem{theorem}{Theorem}[section]
\newtheorem{proposition}[theorem]{Proposition} 
\newtheorem{lemma}[theorem]{Lemma}
 
\newtheorem{definition}[theorem]{Definition} 
\newtheorem{remark}[theorem]{Remark}  
\numberwithin{equation}{section} 

\def \RR   {\mathbb{R}}
\def \RN   {\RR^N}  
\def \del  {\partial}    

\def \Ucal  {\mathcal U}
\def \Vcal  {\mathcal V}
\def \Ecal  {\mathcal E}
\def \lam  {\lambda} 
\def \phif   {\varphi^\flat} 
  
\def \eps  {\epsilon}
\def \lamb {\overline{\lambda}} 
\def \be   {\begin{equation}}
\def \ee   {\end{equation}}
\def \bea  {\begin{eqnarray}} 
\def \eea  {\end{eqnarray}}

\def \bz {\begin{itemize}}
\def \ez {\end{itemize}}

\def\vf    {{\varphi^\flat}}  
\def\vfalpha  {{\varphi^\flat}}   
\def\vfz   {{\varphi^\flat_0}}  
\def\vn    {{\varphi^\natural}}
\def\vs    {{\varphi^\sharp}}
   
\def\ab    {\overline a}   
     
\newcommand \vfpalpha {{\varphi^\flat}}   
\newcommand \anp  {{A^\natural_p}}  
\newcommand{\oalpha}{{{\alpha}}}
\newcommand{\oalphaz}{{{\alpha}}_0} 
\def \Lip   {\text{\rm Lip}}  

\begin{document}

\title{Kinetic relations
\\ 
for undercompressive shock waves.
\\ 
Physical, mathematical, and numerical issues} 
\author{ 
Philippe G. LeFloch 
\footnote{Laboratoire Jacques-Louis Lions \& Centre National de la Recherche Scientifique, 
Universit\'e Pierre et Marie Curie (Paris 6), 4 Place Jussieu, 75252 Paris, France. E-mail : {\sl pgLeFloch@gmail.com.} 
Blog: {\sl http://PhilippeLeFloch.wordpress.com} 
\newline 
2000\textit{\ AMS Subject Classification.} 35l65, 76N10, 35B45. 
}
}
\date{\today}
\maketitle

\begin{abstract} Kinetic relations are required in order to characterize nonclassical undercompressive shock waves 
and formulate a well-posed initial value problem for nonlinear hyperbolic systems of conservation laws. 
Such nonclassical waves arise in weak solutions of a large variety of physical models: 
phase transitions,  thin liquid films, magnetohydrodynamics, Camassa-Holm model,
martensite-auste\-nite materials,  
semi-conductors, combustion theory, etc. This review presents the research done in the last fifteen years which led
the development of the theory of kinetic relations for undercompressive shocks
and has now covered many 
physical, mathematical, and numerical issues. The main difficulty overcome here in our analysis 
of nonclassical entropy solutions comes from their {\sl lack of monotonicity} with respect to initial data. 

First, a nonclassical Riemann solver is determined by imposing a single entropy inequality,
a kinetic relation and, if necessary, a nucleation criterion. 
To determine the kinetic function, the hyperbolic system of equations is augmented with 
diffusion and dispersion terms, accounting for small-scale physical effects such as the viscosity, capillarity, 
or heat conduction of the material under consideration. 
Investigating the existence and properties of traveling wave solutions allows one to establish the existence, 
as well as qualitative properties, of the kinetic function. 
To tackle the initial value problem, 
a Glimm-type scheme based on the nonclassical Riemann solver is introduced, together with 
generalized total variation and interaction functionals which are adapted to nonclassical shocks. 
Next, the strong convergence of the vanishing diffusion-dispersion approximations for the initial value probem 
is established via weak convergence techniques. 
Finally, the numerical approximation of nonclassical shocks relies 
on schemes with controled dissipation, built
from high-order, entropy conservative, finite difference approximations  
and an analysis of their equivalent equations. Undercompressive shocks of hyperbolic conservation laws 
turn out to exhibit features that are very similar to shocks of nonconservative hyperbolic systems, 
who were investigated earlier by the author. 
\end{abstract}

\tableofcontents
  

\section{Introduction}
\label{IN}

\subsection{Augmented systems of conservation laws} 

Nonlinear hyperbolic systems of conservation laws have the general form  
\be 
u_t + f(u)_x = 0, \qquad u=u(t,x), 
\label{IN.system}
\ee 
with unknown $u: \RR^+ \times \RR \to \Ucal$, where
the given flux $f: \Ucal \to \RN$ is defined on a (possibly non-connected) open set $\Ucal\subset \RN$ 
and satisfies the following strict hyperbolicity condition: 
for every $v \in \Ucal$, the matrix $A(v):=Df(v)$ admits $N$ real and distinct eigenvalues 
$$
\lam_1(v) < \ldots < \lam_N(v)
$$
and a basis of right-eigenvectors $r_1(v), \ldots, r_N(v)$.  
It is well-known that singularities arise in finite time in initially smooth solutions to \eqref{IN.system},
which motivates one to seek for weak solutions understood in the sense of distributions, containing
for instance shock waves.  
It turns out that weak solutions are not uniquely determined by their initial data at $t=0$, say, 
and consequently, 
it is necessary to impose an ``entropy condition'' in order to formulate a well-posed initial value problem. 
 
In the present work, the weak solutions of interest are realizable as limits ($\eps \to 0$) of 
smooth solutions ${u^\eps = u^\eps(t,x)}$ to an {\sl augmented model} of the form 
\be 
u^\eps_t + f(u^\eps)_x = R^\eps_x := \big( R(\eps \, u^\eps_x, \eps^2 \, u^\eps_{xx}, \ldots) \big)_x. 
\label{IN.augmented}
\ee
The parameter $\eps>0$ introduces a small-scale in the problem, and the term $R_x^\eps$ 
accounts for high-order physical features, 
neglected at the hyperbolic level of modeling \eqref{IN.system}. 
In the examples arising in fluid dynamics and material science, 
suitable restrictions arise on the right-hand side $R^\eps_x$ so that, for instance,   
the conservation laws \eqref{IN.system} are recovered in the limit $\eps \to 0$.

Fine properties of the physical medium under consideration are taken into account in \eqref{IN.augmented}. 
For instance, \eqref{IN.system} might be the Euler equations for liquid-vapor mixtures, 
while its augmented version \eqref{IN.augmented} is
a Korteweg-type model including both viscosity and capillarity effects. 
In many examples of interest, the augmented system admits global-in-time smooth solutions 
that are uniquely determined by their initial data. In this context, 
it is natural to select those solutions of \eqref{IN.system} that are realizable as (possibly only formal) limits 
of solutions to \eqref{IN.augmented}, that is, to define 
\be
u := \lim_{\eps \to 0} u^\eps. 
\label{IN.formal}
\ee  
We face here the fundamental issue raised in the present review: 
How may one characterize this limit $u$~?
In particular, which propagating discontinuities should be considered admissible~? 
Our primary focus is thus on the derivation of the ``sharp interface'' theory associated with 
\eqref{IN.augmented} consisting of 
finding algebraic conditions for shock waves which, 
within the class of weak solutions with bounded variation, 
single out a unique solution to the initial value problem 
associated with \eqref{IN.system} (or, at least, a unique solution to the Riemann problem).

A large part of the mathematical literature on shock waves is focused on 
small amplitude solutions generated by the vanishing viscosity method, 
corresponding in \eqref{IN.augmented} to the regularization 
\be
\label{visco}
R^\eps_x :=  \eps \, u^\eps_{xx}. 
\ee
With this regularization, the limits $u := \lim_{\eps \to 0} u^\eps$ are characterized uniquely 
\cite{Volpert,Kruzkov,BressanLeFloch, BLP-unique,BianchiniBressan}   
by one of the equivalent formulations of the ``classical'' entropy condition introduced 
by Lax \cite{Lax}, Oleinik \cite{Oleinik}, Kruzkov \cite{Kruzkov}, 
Wendroff \cite{Wendroff}, Dafermos \cite{Dafermos-entropy, Dafermos-book}, 
and Liu \cite{Liu1,Liu2}. 
Following \cite{LeFloch-book}, we refer to these weak solutions as {\sl classical entropy solutions.}


\subsection{Diffusive-dispersive approximations}

Our objective is to develop a theory of weak solutions that encompasses regularizations 
beyond \eqref{visco}, including {\sl diffusive-dispersive regularizations} of the form  
\be
R^\eps_x = \eps \, u^\eps_{xx} + \alpha \, \eps^2 \, u^\eps_{xxx}. 
\label{IN.diffusivedispersive}
\ee
Here, $\alpha$ is a fixed real parameter and $\eps \to 0$, 
so that the two regularization terms are {\sl kept in balance}  
and a subtle competition takes place between two distinct phenomena: 
the diffusion term $\eps \, u^\eps_{xx}$ has a regularizing effect on shock waves 
while the dispersion term $\alpha \, \eps^2 \, u^\eps_{xxx}$ generates oscillations with high frequencies. 
However, as is for instance observed in numerical experiments, 
such oscillations arise near the jump discontinuities, only, and form spikes 
with finite amplitude  
superimposed on shocks. In turn, the formal limit \eqref{IN.formal} does exist
 but do not coincide with the one selected by the viscosity approximation \eqref{visco}. 
 In fact, these limiting solutions {\sl depend} on the value of the parameter $\alpha$, and    
the same initial data with {\sl different} values of $\alpha$
give rise to {\sl different} shock wave solutions.   

The above observation motivates us to develop, for the system of conservation laws  \eqref{IN.system},
{\sl several theories of shock waves,} each being associated with a different formulation of the 
entropy condition. In the present work, we are primarily interested in hyperbolic 
systems  
for which the characteristic speed $\lambda_j = \lambda_j(u)$ along each integral curve of the vector field
$r_i=r_j(u)$ admits {\sl one extremum point,} at most. 
For such systems and the class regularization \eqref{IN.diffusivedispersive},
a theory of entropy solution is now available which is based on 
the concept of a {\sl kinetic relation.} 
The basic strategy is to impose a single entropy inequality in the sense of distributions (cf.~\eqref{GT.entropy}, below), together with an additional algebraic condition imposed on certain {\sl nonclassical undercompressive} shock waves.

The kinetic relation was first introduced and developed in the context of 
the {\sl dynamics of materials undergoing phase transitions,}  
described by the system of two conservation laws 
\be
\label{555}
\aligned 
 w_t - v_x &  = 0,
\\
v_t - \sigma(w)_x 
&  = \eps \, v_{xx} - \alpha \, \eps^2 \, w_{xxx},  
\endaligned    
\ee
where $v$ denotes the velocity, $w>-1$ the deformation gradient, and $\sigma=\sigma(w)$ 
the stress of the material. The parameters $\eps$ and $\alpha \, \eps^2$ 
represent the (rescaled) viscosity and capillarity of the material. 
The mathematical research on \eqref{555} began with works by Slemrod in 1984 on self-similar solutions 
to the Riemann problem \cite{Slemrod1,Slemrod2} (i.e.~the initial value problem with piecewise constant data), 
and Shearer in 1986 on the explicit construction of a Riemann solver when $\alpha = 0$ \cite{Shearer-86}.
The notion of a kinetic relation for {\sl subsonic phase boundaries} 
was introduced by Truskinovsky in 1987 \cite{Truskinovsky1,Truskinovsky2,Truskinovsky3}, 
and Abeyaratne and Knowles in 1990 \cite{AbeyaratneKnowles1,AbeyaratneKnowles2}, who 
solved the Riemann problem for \eqref{555} and investigated the existence of traveling wave solutions
when $\sigma$ is a piecewise linear function. 
In 1993, LeFloch \cite{LeFloch-ARMA}
introduced a mathematical formulation of the kinetic relation for \eqref{555}
within the setting of functions of bounded variation and tackled the initial value problem via the Glimm scheme;
therein, the kinetic relation was interpreted as an {\sl entropy dissipation measure} 
(cf.~Section~\ref{GT}, below). 
The kinetic relation was then extended to general hyperbolic systems by the author 
together with many collaborators,    
and developed into a general {\sl mathematical theory of nonclassical shocks,}
which we review 
in the present notes. For earlier reviews, see \cite{LeFloch-Freiburg,LeFloch-book}.


\subsection{Main objectives of the theory}

Weak solutions in the sense of distributions can be defined within, for instance, 
the space  $L^\infty$ of bounded and measurable functions.
However, to deal with nonclassical solutions one needs to prescribe (pointwise) algebraic conditions 
at jumps and, consequently, one must work within a space of functions admitting traces such as  
the space of functions with bounded variation. In turn, 
the theory of nonclassical shocks relies on a {\sl pointwise formulation} of the notion of entropy solution
and on techniques of pointwise convergence. The main difficulty overcome here in our analysis 
of nonclassical entropy solutions comes from their {\sl lack of monotonicity} with respect to initial data. 
 The following issues have been addressed.

\bz

\item[1.] {\sl Derivation and analysis of physical models.}

Such nonclassical undercompressive waves arise in weak solutions to a large variety of physical models
from fluid or solid dynamics: 
phase transitions,  thin liquid films, magnetohydrodynamics, 
Camassa-Holm model,
martensite-auste\-nite materials, 
semi-conductors, combustion theory, etc.

\item[2.] {\sl Mathematical theory of nonclassical entropy solutions.}
 
\begin{itemize}

\item By considering initial data consisting of two constant states separated by a single discontinuity, 
one constructs a {\sl nonclassical Riemann solver} associated with a prescribed {\sl kinetic function,}  
compatible with an {\sl entropy inequality.}   

\item Kinetic functions can be determined from traveling wave solutions to an augmented model like 
\eqref{IN.augmented}-\eqref{IN.diffusivedispersive} and, for certain models,  
monotonicity and asymptotic properties of the kinetic functions can be established.

\item Establishing via Glimm-type schemes
the existence of nonclassical entropy 
solutions to the initial value problem associated with \eqref{IN.system}  
requires a uniform estimate on the total variation of solutions. 
Such a bound is derived with the help of generalized variation functionals that are designed to be 
diminishing in time. 

\item The zero diffusion-dispersion limits for the initial value problem is justified rigorously by 
using weak convergence techniques:
compensated compactness, measure-valued solutions, kinetic formulation. 
 
\ez

\item[3.] {\sl Approximation of nonclassical entropy solutions.}  

\bz

\item Finite difference {\sl schemes with controled dissipation} are constructed from
entropy conservative flux-functions and a careful analysis of their equivalent equations. 

\item Kinetic functions are associated with finite difference schemes and computed numerically.   
Limiting solutions depend on the regularization terms arising 
in the augmented model and, for a given initial value problem, different solutions 
are obtained if different viscosity-capillarity ratio or different discretization schemes are used.  

\ez

\item[4.] {\sl Applications and limitations of the theory.}  

\bz

\item Nonclassical solutions exhibit particularly complex wave structures, 
and 
the kinetic function appears to be the proper tool to represent the entire dynamics of nonclassical shocks.

\item Of course, describing all singular limits of a given augmented model  
via a ``purely'' hyperbolic theory need not always be possible. For certain regularizations for which the 
traveling wave analysis does not lead to a unique Riemann solver, 
a {\sl nucleation criterion} may be required in order to uniquely characterize nonclassical entropy solutions.  

\ez

\ez

One can not underestimate the fruitful connection
between the theory of kinetic relations for undercompressive shocks and 
the so-called DLM (Dal~Maso, LeFloch, Murat) theory of nonconservative hyperbolic systems. 
This theory was introduced and developed in 
\cite{LeFloch-CPDE,LeFloch-IMA,DLM,LeFlochLiu}. 
Fundamental numerical issues were first discussed in \cite{HouLeFloch}; 
for recent progress, see \cite{BerthonCoquel,BerthonCoquelLeFloch,Castro--Pares,ChalonsCoquel} and the references therein.

Building on the pioneering papers \cite{AbeyaratneKnowles1,LeFloch-ARMA,Truskinovsky1}, 
the research on undercompressive shocks developed intensively in the last fifteen years, 
and the author of this review 
is very grateful to his collaborators, postdocs, and students who accompanied him on this subject, 
including B.T. Hayes (basic theoretical and numerical issues),  
N. Bedjaoui (traveling wave solutions), K.T. Joseph (self-similar approximations), 
M.-D. Thanh (Riemann problem), M. Shearer (nucleation criterion), 
D. Amadori, P. Baiti, M. Laforest, B. Piccoli (Glimm-type methods), 
J.C. Correia, C. Kondo (vanishing diffusive-dispersive limits), 
F. Boutin, C. Chalons, F. Lagouti\`ere, J.-M. Mercier, {S. Mishra}, M. Mohamadian, and 
C. Rohde (numerical methods).


\section{Kinetic relations for undercompressive shocks}
\label{GT} 

\subsection{A single entropy inequality} 
\label{GT-EI} 

We follow the strategy advocated in LeFloch \cite{LeFloch-ARMA,LeFloch-book} for the analysis of
the formal limits \eqref{IN.formal} associated with a given augmented model.
Recall that we are interested in deriving an entropy condition that singles out all 
solutions to \eqref{IN.system} realizable as limits of smooth solutions to \eqref{IN.augmented}. 

In the applications, \eqref{IN.augmented} is ``compatible''  (in a sense defined below) 
with one particular mathematical entropy of the hyperbolic system \eqref{IN.system}.
Hence, we assume that \eqref{IN.system} is endowed with a {\sl strictly convex entropy pair,}
denoted by $(U,F)$. (Strict convexity is imposed throughout this presentation, but 
can be relaxed on certain examples.) 
By definition, $(U,F) : \Ucal \to \RR \times \RN$ is a smooth map such that 
$$
\del_j F = \del_k U \, \del_j f^k, \qquad j=1, \ldots, N, 
$$
where $\del_j$ denotes a partial derivative with respect to the conservative variable $u_j$, 
with $u=(u_1, \ldots, u_N)$ and $f=(f^1, \ldots, f^N)$, and 
we use implicit summation over repeated indices. 

It is easily checked from the definition of an entropy 
that every {\sl sufficiently regular} solution to \eqref{IN.system} satisfies the additional conservation law
\be
U(u)_t + F(u)_x = 0. 
\label{GT.equality}
\ee

\begin{definition}  
\label{fff}
An augmented version \eqref{IN.augmented} of a system of conservation laws \eqref{IN.system} 
endowed with a strictly convex entropy pair $(U,F)$, 
is said to be {\rm conservative} 
and {\rm dissipative with respect to the entropy $U$} 
if 
for every non-negative, compactly supported, smooth function $\theta=\theta(t,x)$ 
\be
\lim_{\eps \to 0} \iint_{\RR^+\times\RR} R^\eps \, \theta  
\, dtdx =0,  
\label{GT.conservative}
\ee
and 
\be 
\lim_{\eps \to 0} \iint_{\RR^+\times\RR} \nabla U(u^\eps) \cdot R^\eps \, \theta  
\, dtdx  \leq 0,   
\label{GT.dissipative}
\ee 
respectively. 
\end{definition} 

The conditions in Definition~\ref{fff} are fulfilled by many examples
of interest and, when the singular limit \eqref{IN.formal} exists in a strong topology, 
they lead to the conservation laws \eqref{IN.system} and 
an entropy inequality.

\begin{definition} 
\label{GT-entropysolution}
Suppose that \eqref{IN.system} is a system of conservation laws endowed with a 
strictly convex entropy pair $(U,F)$. A bounded and measurable function with locally bounded variation 
$u=u(t,x) \in \Ucal$ is called an {\rm entropy solution} if the equations \eqref{IN.system}  
together with the entropy inequality 
\be
U(u)_t + F(u)_x  \leq 0, 
\label{GT.entropy}
\ee
hold in the sense of distributions.  
\end{definition}

This should be regarded as a preliminary definition of solution since, in general, 
\eqref{GT.entropy} is {\sl not} sufficient to guarantee uniqueness for the initial value problem, 
and we are going to impose additional constraints.

The case of scalar, {\sl convex} flux-functions is comparatively simpler: {\sl one}
 entropy inequality turns out to be 
sufficient to characterize a unique limit (see Panov \cite{Panov-unique}
and De~Lellis, Otto, and Westdickenberg \cite{DOW}) 
and, consequently, shock waves satisfying a single entropy inequality 
are regularization-independent. 
The corresponding solutions are referred to as {\sl classical entropy solutions} and contain
compressive shocks, only, which satisfy Lax's shock inequalities. 

For strictly hyperbolic systems admitting characteristic fields  that are 
either {\sl genuinely nonlinear} ($\nabla \lam_j \cdot r_j \neq 0$) 
or {\sl linearly degenerate}  ($\nabla \lam_j \cdot r_j \equiv 0$),
 the inequality \eqref {GT.entropy} for a given strictly convex entropy
 is sufficiently discriminating to select a unique entropy solution. 
This fact was established first for the Riemann problem in Lax's pioneering paper \cite{Lax57}.  
For the Cauchy problem associated with systems with genuinely nonlinear or linearly degenerate 
characteristic fields, a general uniqueness theorem 
in a class of functions with bounded variation 
was established by Bressan and LeFloch \cite{BressanLeFloch}. 
Hence, weak solutions to genuinely nonlinear or linearly degenerate systems are {\sl independent} 
of the precise regularization mechanism $R^\eps$ in the right-hand side of \eqref{IN.augmented}, 
as long as it is conservative and dissipative in the sense of Definition~\ref{fff}.


\subsection{Kinetic relation. Formulation based on state values} 
\label{GT-RP} 

In the applications, many models arising in continuum physics (see Section~\ref{MOD}, below)
do not have globally genuinely nonlinear characteristic fields. 
In numerical experiments with such systyems, weak solutions often exhibit particularly complex wave patterns, 
including {\sl undercompressive shocks.}  
Distinct solutions are obtained for the same initial value problem 
if one changes the diffusion-dispersion ratio, the regularization, or the approximation scheme.  
From the analysis standpoint, it turns out that
for such systems one entropy inequality \eqref {GT.entropy} is not 
sufficiently discriminating: the initial value problem admits 
a large class of entropy solutions (satisfying a single entropy inequality). 
Weak solutions are highly {\sl sensitive} to small-scales  
neglected at the hyperbolic level of physical modeling and one needs to determine further admissibility conditions beyond \eqref{GT.entropy}. No {\sl universal} admissibility criterion is available in this context but, instead, 
{\sl several hyperbolic theories} must be developed, each being determined by  
specifying a physical regularization.  The approach advocated by the author is based on 
imposing a {\sl kinetic relation} in order to uniquely characterize the dynamics of nonclassical undercompressive shocks.
   
Recall that a {\sl shock wave} $(u_-,u_+)$ is a step function 
connecting two constant states $u_-, u_+$ through a single jump discontinuity,
propagating
at some finite speed $\lamb=\lamb(u_-,u_+)$. To be an entropy solution in the sense of Definition~\ref{GT-entropysolution}, 
$(u_-, u_+)$ must satisfy the {\sl Rankine-Hugoniot relation} 
\be
- \lamb \, (u_+ - u_-) + f(u_+) - f(u_-) = 0, 
\label{GT.Hugoniot}
\ee 
as well as the {\sl entropy inequality}  
\be 
- \lamb \, \bigl(U(u_+) - U(u_-) \bigr) + F(u_+) - F(u_-) \leq 0. 
\label{GT.entropyineq} 
\ee
Denote by $\Ecal^U$ the set of all pairs $(u_-, u_+) \in \Ucal \times \Ucal$ 
satisfying \eqref{GT.Hugoniot}-\eqref{GT.entropyineq} for some speed $\lamb=\lamb(u_-, u_+)$.

The emphasis in this section is on solutions to strictly hyperbolic systems with 
small amplitude (but the presentation extends to large data for particular examples). 
Then, one can show that, for each $u_-$,
\eqref{GT.Hugoniot} consists of $N$ curves issuing from $u_-$ and tangent to each characteristic vector $r_j(u_-)$. 
Wen the shock speed
$\lamb=\lamb(u_-, u_+)$ is comparable to the characteristic speeds $\lamb_j(u_\pm)$, 
we write also $\lamb=\lamb_j(u_-, u_+)$ and accordingly we decompose the shock set
$$
\Ecal^U = \Ecal_1^U \cup \ldots \cup \Ecal_N^U. 
$$
Furthermore, we tacitly assume that the set of definition $\Ucal$ is replaced by a smaller open subset, if necessary.  
Finally, we assume that  the characteristic speed $\lambda_j = \lambda_j(u)$ along each integral curve of the vector field
$r_i=r_j(u)$, or along each Hugoniot curve, admits {\sl one extremum point,} at most. We refer to characteristic fields
having such an extremum as {\sl concave-convex characteristic fields.}

A $j$-shock $(u_-, u_+)$ is called {\sl slow undercompressive} if 
$$
\lambda_j(u_\pm) > \lamb_j(u_-, u_+)
$$
and {\sl fast undercompressive} if the opposite inequalities hold. 
Across undercompressive shocks, one must supplement 
the Rankine-Hugoniot relation and entropy inequality
 with an additional jump condition, as follows. 
Recall that functions of bounded variation (BV) admit traces in a measure-theoretic sense.   
While \eqref{GT.equality} and \eqref{GT.entropy} can be imposed in the sense of distributions, 
this regularity of BV functions is here
required for pointwise conditions to make sense.

\begin{definition}[Formulation based on state values]
Consider small amplitude solutions to a strictly hyperbolic
system of conservation laws \eqref{IN.system} 
endowed with a strictly convex entropy pair $(U,F)$ and admitting genuinely nonlinear, linearly degenerate, or concave-convex characteristic fields, only.  

1. A {\rm kinetic function compatible with the entropy $U$} for a concave-convex $j$-characteristic family 
is a Lipschitz continuous map $\phif_j=\phif_j(u)$ defined over $\Ucal$ such that 
\be
(u, \phif_j(u)) \in \Ecal_j^U \, \text{ is an undercompressive shock.} 
\label{GT.mapPhi}
\ee

2. A family of kinetic functions  $\phif := (\phif_j)$ being prescribed for each concave-convex characteristic field, 
one says that 
a bounded and measurable function with locally bounded variation $u=u(t,x) \in \Ucal$ 
is a {\rm $\phif$-admissible entropy solution} if it is an entropy solution to 
\eqref{IN.system}-\eqref{GT.entropy} and,
 at every point of approximate jump discontinuity $(t,x)$ of $u$ associated with an undercompressive
 $j$-shock, 
\be
u_+ = \phif_j(u_-), 
\label{GT.kineticrelation299}
\ee 
where $u_-, u_+$ are the left- and right-hand limits at that point. 
\end{definition}

We will refer to \eqref{GT.kineticrelation299} as the {\sl kinetic relation} associated with $\phif_j$, 
and 
to $\phif$-admissible entropy solutions as {\sl nonclassical entropy solutions.} 
Note that not all propagating waves within a solution 
require a kinetic relation, but only undercompressive shocks do.
The kinetic relation is very effective and, for many (but not all) models of interest, 
selects a unique $\phif$-admissible entropy solution to the Riemann problem.


\subsection{Kinetic function. Formulation based on the entropy dissipation} 
\label{GT-KF} 

To have an effective theory, the kinetic function must be specified. One may simply postulate
the existence of the kinetic function $\phif$ without refereeing to a small-scale modeling, but instead 
defining it via laboratory experiments or some physical heuristics. 
A sounder approach relies on a prescribed augmented model, as we now explain. 

We introduce the  kinetic function in the following way  
based on the entropy dissipation measure generated by a given augmented model. 
Suppose that the product $\nabla U(u^\eps) \cdot R^\eps$ arising in the right-hand side of 
\eqref{IN.augmented} can be decomposed in the form 
$$
\nabla U(u^\eps) \cdot R^\eps = Q^\eps - \mu^\eps,
$$ 
where $Q^\eps$ converges to zero in the sense of distributions and 
$\mu^\eps$ is a uniformly bounded sequence of non-negative $L^1$ functions.
Let us refer to $\mu^\eps$ as the {\sl entropy dissipation measure.} 
Note that it depends on the specific regularization $R^\eps$ and the entropy $U$. 
(This decomposition can be established for several examples; see Section~\ref{MOD}.)
After extracting a subsequence if necessary, this sequence of measures converges  
in the weak-star sense 
$$ 
\mu^U (u) := \lim_{\eps \to 0} \mu^\eps \leq 0,
$$ 
where the limiting measure $\mu^U (u)$ is clearly related to the pointwise limit 
$u:=\lim_{\eps \to 0} u^\eps$, but in general {\sl cannot} be uniquely determined from 
the sole knowledge of this limit. 
 
For regularization-independent shock waves the sole {\it sign\/} 
of the entropy dissipation measure $\mu^U (u)$ suffices and one simply writes down the 
entropy inequality \eqref{GT.entropy}. However, for regularization-sensitive shock waves, 
the {\it range} of the measure $\mu^U (u)$ plays a crucial role in selecting  
weak solutions. 
It is proposed in LeFloch~\cite{LeFloch-ARMA,LeFloch-Freiburg,LeFloch-book} to replace the entropy inequality \eqref{GT.entropy}
by the {\sl entropy equality} 
\be
\label{pf}
U(u)_t + F(u)_x  = \mu^U(u) \leq 0, 
\ee
where $\mu^U(u)$ is a non-positive, locally bounded measure depending on the solution 
$u$ under consideration. 
Clearly, the measure $\mu^U(u)$ cannot be prescribed arbitrarily and, in particular,  
must vanish on the set of continuity points of $u$. 
  
Recalling that we focus the presentation on strictly hyperbolic systems and small data,
let $\Lambda_j$ be the range of the speed function $\lambda_j$, and let us use the short-hand
notation $\Ecal^U \times \Lambda$ for the union of all subsets $\Ecal^U_j \times \Lambda_j$.

\begin{definition}[Formulation based on the entropy dissipation measure]
  Suppose that \eqref{IN.system} is a system of conservation laws endowed with an entropy pair $(U,F)$. 

1. A {\rm kinetic function (compatible with the entropy $U$)} 
is a Lipschitz continuous map $\Phi=\Phi(u_-, u_+, \lam)$ defined on $\Ecal^U \times \Lambda$ 
and satisfying
\be
\Phi \leq 0. 
\label{GT.mapPhi2}
\ee

2. Let $\Phi$ be a kinetic function. A bounded and measurable function with locally bounded variation 
$u=u(t,x) \in \Ucal$ is called a {\rm $\Phi$-admissible entropy solution} if it is an entropy solution of 
\eqref{IN.system}-\eqref{GT.entropy} and if at every point of approximate jump discontinuity $(t,x)$ of $u$ 
\be
- \lam \, \bigl(U(u_+) - U(u_-) \bigr) + F(u_+) - F(u_-) = \Phi(u_-, u_+, \lam),   
\label{GT.kineticrelation1}
\ee
where $u_-, u_+,\lam$ denote the left- and right-hand limits and $\lam$ the shock speed. 
\end{definition}

The entropy dissipation \eqref{GT.kineticrelation1} is analogous to what is called 
a ``driving force'' in the physical literature. 
Note that at a point $(u_-,u_+)$ where the following trivial choice is made 
\be 
\Phi(u_-, u_+, \lam) := - \lam \, \bigl(U(u_+) - U(u_-) \bigr) + F(u_+) - F(u_-),   
\label{trivialchoice} 
\ee
the condition \eqref{GT.kineticrelation1} is vacuous. In practice, 
$\Phi(u_-, u_+, \lam)$ will be strictly monotone in $\lambda$ for all undercompressive shocks,
and \eqref{GT.kineticrelation1} may be seen to select a unique propagation speed $\lambda$
for each $u_-$. (See \cite{LeFloch-book} for details.)

To determine the kinetic function,  one then analyzes traveling wave solutions associated with \eqref{IN.augmented}. 
Given any  shock wave $(u_-, u_+)$, a function $u^\eps(t,x) = w(y)$ with 
$y := (x- \lam \, t)/\eps$ is called a {\sl traveling wave} associated with the shock 
$(u_-, u_+)$
if it is 
a smooth solution to \eqref{IN.augmented} satisfying 
\be
\label{un1}
w(-\infty) = u_-, \quad w(+\infty) = u_+, 
\ee
and 
\be
\label{deux2}
\lim_{|y| \to +\infty} w'(y) = 
\lim_{|y| \to + \infty} w''(y) = \ldots =0.  
\ee
The function $w$ satisfies the following system of ordinary differential equations 
\be 
R(w', w'', \ldots) = - \lambda ( w - u_-) + f(w) - f(u_-). 
\label{trois3}
\ee
At this juncture, it is straightforward, but fundamental, to observe the following pro\-perties.

\begin{proposition}[Kinetic relations derived from traveling waves]
\label{Tra} 
Consider an augmented version \eqref{IN.augmented} of a system of conservation laws \eqref{IN.system} 
which is 
endowed with a strictly convex entropy pair $(U,F)$, and is 
conservative and dissipative with respect to $U$.  Then, 
given  a traveling wave solution satisfying \eqref{un1}--\eqref{trois3}, the pointwise limit  
$$ 
u(t,x) 
: = \lim_{\eps \to 0} w\Big({x - \lam \, t \over \eps} \Big) 
= 
\begin{cases}
 u_-, &   x< \lam \, t, 
\\
       u_+, &   x> \lam \, t, 
\end{cases} 
$$
is a weak solution to \eqref{IN.system} satisfying the entropy inequality \eqref{GT.entropy}. 
In particular, the Rankine-Hugoniot relation \eqref{GT.Hugoniot} follows
by letting $y \to +\infty$. 

Moreover, the solution $u$ satisfies the kinetic relation \eqref{GT.kineticrelation1}
in which the dissipation measure reads 
$$
\aligned 
& \mu^U(u) = M^U[u] \, \delta_{x- \lam \, t}, 
\\  
& M^U[u] := - \int_\RR R(w(y), w'(y), w''(y), \ldots) \cdot \nabla^2 U(w) \, w'(y) \, dy,  
\endaligned 
$$ 
where $\delta_{x- \lam \, t}$ denotes the Dirac measure concentrated on the line
${x - \lam \, t =0}$. 
\end{proposition}

         
\section{Physical models} 
\label{MOD}

\subsection{Nonlinear diffusion model} 

We begin with the case of a conservation law regularized with diffusion, i.e. 
\be
\label{356}
u^\eps_t + f(u^\eps)_x = \eps \, \bigl( b(u^\eps) \, u^\eps_x\bigr)_x, \qquad u^\eps=u^\eps(t,x) \in \RR, 
\ee
where the function $b: \RR \to (0, \infty)$ is bounded above and below, and $\eps > 0$ is a small parameter.  
The following observation goes back to Kruzkov \cite{Kruzkov} and Volpert \cite{Volpert}
and, in particular, given suitable initial data, the solutions to the initial value problem 
associated with \eqref{356}
converge strongly as $\eps \to 0$.

\begin{lemma}[Nonlinear diffusion model]
Solutions $u^\eps$ to the augmented model \eqref{356} satisfy, for every convex function $U: \RR \to \RR$, 
$$
\aligned
& U(u^\eps)_t  + F(u^\eps)_x = - D^\eps + \eps \, C^\eps_x.
\\
& D^\eps := \eps \, b(u^\eps) \, U''(u^\eps) \, |u^\eps_x|^2, 
 \qquad C^\eps := b(u^\eps) \, U(u^\eps)_x, 
\endaligned
$$
in which $F(u) := \int^u f'(v) \, U'(v) \, dv$.  
Hence, the limit $u = \lim_{\eps \to 0} u^\eps$  satisfies the 
following {\rm entropy inequalities associated with the nonlinear diffusion model} 
\be
\label{357}
U(u)_t + F(u)_x \leq 0.
\ee
\end{lemma}

We refer to weak solutions $u \in L^\infty$ 
to the hyperbolic conservation law 
\be
\label{hyp1} 
u_t + f(u)_x = 0, \qquad u=u(t,x) \in \RR
\ee
that satisfy {\sl all} of the inequalities \eqref{357} as {\sl classical entropy solutions.} 
Clearly, the inequalities \eqref{357} are equivalent to the Kruzkov inequalities 
$$
|u-k|_t + \Big( \text{sgn}(u-k) (f(u) - f(k)) \Big)_x \leq 0, \qquad k \in \RR. 
$$
Provided additional regularity beyond $L^\infty$ is assumed and, for instance, 
when $u$ has bounded variation,
the classical entropy solutions to \eqref{hyp1} 
can be also characterized by the Oleinik inequalities (cf.~\ref{129}, below),
which, at shocks, impose a local concavity or convexity 
property of $f$. 


\subsection{Linear diffusion-dispersion model} 

Nonclassical solutions are obtained when both diffusive and dispersive effects are taken into account. 
A typical model of interest is provided by the following conservation law 
\be
\label{cons2}
u^\eps_t + f(u^\eps)_x = \eps \, u^\eps_{xx} + \gamma(\eps) \, u^\eps_{xxx}, \qquad u^\eps= u^\eps(t,x), 
\ee 
in which $\eps >0$ and $\gamma=\gamma(\eps)$ are real parameters tending to zero.  
This equation was studied first by Shearer et al. \cite{JMS}, Hayes and LeFloch \cite{HayesLeFloch-scalar}, 
and Bedjaoui and LeFloch \cite{BedjaouiLeFloch1}.

The relative scaling between $\eps$ and $\gamma(\eps)$ determines the limiting behavior of the regularized solutions, 
as follows:  
\begin{itemize} 

\item When $\gamma(\eps) << \eps^2$, the diffusion plays a dominant role and the limit
$u:= \lim_{\eps \to 0} u^\eps$ is a classical entropy solution.

\item When $\gamma(\eps) >> \eps^2$, the dispersion effects are dominant and high oscillations develop in the limit, 
allowing only for weak convergence of $u^\eps$. The theory for
 dispersive equations developed by Lax and Levermore \cite{LL} is the relevant theory in this regime.

\item Finally, in the balanced regime where $\gamma(\eps) := \alpha \, \eps^2$ for fixed $\alpha$, 
the limit $u:= \lim_{\eps \to 0} u^\eps$ exists in a strong sense 
and solely mild oscillations arise near shocks. The limit $u$ is a weak solution to the hyperbolic 
conservation law \eqref{hyp1}. 
When $\alpha >0$, 
the limit $u$ exhibits a nonclassical behavior, and strongly depends on $\alpha$.   
\end{itemize}

The above observations motivate us in the rest of this paper, and we adopt the following scaling
\be
\label{cons21}
u^\eps_t + f(u^\eps)_x = \eps \, u^\eps_{xx} + \alpha \, \eps^2 \, \, u^\eps_{xxx}, 
\qquad u_\eps = u^\eps= u^\eps(t,x),
\ee 
where $\alpha$ is fixed. 
Following Hayes and LeFloch \cite{HayesLeFloch-scalar}, we now derive {\sl one} entropy inequality in the limit.
The contribution due to the diffusion decomposes into a non-positive term and a conservative one; 
the contribution due to the dispersion is entirely conservative;
hence, formally at least, as $\eps \to 0$ we recover the entropy inequality.

\begin{lemma}[Linear diffusion-dispersion model]
\label{DDD} 
For the augmented model \eqref{cons2}, one has 
$$
\aligned
&  (1/2) \, | u_\eps^2 |_t  + F(u^\eps)_x
= - D^\eps + \eps \, C^\eps_x, 
\\
& D^\eps := \eps \, |u^\eps_x|^2 \geq 0, 
\qquad
C^\eps := u^\eps u^\eps_x + \alpha \, \eps \, \big(u^\eps \, u^\eps_{xx} - (1/2) \, |u^\eps_x|^2 \big). 
\endaligned 
$$
Hence, as $\eps \to 0$, the (possibly formal) limit  $u:= \lim_{\eps \to 0} u^\eps$ 
satisfies the single {\rm entropy inequality associated with the linear diffusion-dispersion model}
$$
\bigl( u^2/2 \bigr)_t + F(u)_x \leq 0, \qquad F' := u \, f'. 
$$
No specific sign is available for arbitrary convex entropies. 
\end{lemma}

More generally, consider the nonlinear diffusion-dispersion model 
\be
u^\eps_t + f(u^\eps)_x
= \eps \, \bigl(b(u^\eps) \, u^\eps_x \bigr)_x + \alpha \, \eps^2 \, \big(
c_1(u^\eps) \, \big( c_2(u^\eps) \, u^\eps_x \big)_x \big)_x,
\label{CL.diffusivedispersive}
\ee
where the functions $b, c_1, c_2$ are given smooth and positive functions, and 
$\alpha$ is a real parameter.

\begin{lemma}[Nonlinear diffusion-dispersion model]
For the augmented model \eqref{CL.diffusivedispersive}, 
the formal limit $u=\lim_{\eps \to 0} u^\eps$ 
satisfies the following {\rm entropy inequality
 associated with the nonlinear diffusion-dispersion model} 
$$
\aligned
& U(u)_t + F(u)_x \leq 0, 
\\
& U'' = {c_2 \over c_1} >0, \qquad F' := f' \, U'. 
\endaligned
$$ 
\end{lemma} 

\begin{proof} In the entropy variable $\hat u = U'(u)$, the dispersive term takes the 
symmetric form 
$$
\big(c_1(u) \, \big( c_2(u) \, u_x \big)_x\big)_x
= \big( c_1(u) \, \big( c_1(u) \, \hat u_x \big)_x \big)_x 
$$
and, consequently, any solution of \eqref{CL.diffusivedispersive} satisfies  
$$
\aligned
& U(u^\eps)_t  + F(u^\eps)_x = - D^\eps + \eps \, C^\eps_x.
\\
& D^\eps := \eps \, b(u) \, U''(u) \, |u_x|^2, 
\\
& C^\eps :=  b(u) \, U'(u) \, u_x 
           + \alpha\eps \, \bigl(c_1(u) \hat u \, \big( c_1(u) \, \hat u_x \big)_x  - | c_2(u) \, u_x|^2 /2 \bigr). 
\endaligned 
$$
\end{proof}


\subsection{Thin liquid film model} 

Consider next the augmented model 
\be
\label{thin}
u^\eps_t + (u_\eps^2 - u_\eps^3)_x 
= 
\eps \, (u_\eps^3 u_x^\eps)_x - \gamma(\eps) \, ( u_\eps^3 \, u^\eps_{xxx} )_x
\ee
with $\eps, \gamma=\gamma(\eps) >0$, 
in which the right-hand side describes the effects of surface tension on a thin liquid film moving on a surface. 
In this context,  $u=u(t,x) \in [0,1]$ denotes the normalized thickness of the thin film layer, 
while parameters governing the various forces and the slope of the surface are typically incorporated into 
the parameter 
$\eps$. 
More precisely, the equation \eqref{thin} arises from the  lubrication aproximation of the Navier-Stokes equation 
and models the physical situation in which the film is driven by two counteracting forces: 
on one hand, the gravity is responsible for pulling the film down an inclined plane
and a thermal gradient (resulting from the surface tension gradient) pushing the film up the plane. 
The thin liquid film model was studied by Bertozzi and Shearer \cite{BertozziShearer}, together with  
M\"unch \cite{BertozziMunchShearer}, Levy \cite{LevyShearer1,LevyShearer2}, and 
Zumbrun \cite{BMSZ}. See also 
Otto and Westdickenberg \cite{OttoWestdickenberg}, 
and LeFloch and Mohamadian \cite{LeFlochMohamadian}.

For the purpose of the present paper we need the following observation, made in \cite{LeFlochShearer}. 

\begin{lemma}[Thin liquid film model]
For the augmented model \eqref{thin}, one has 
$$
( u^\eps \log u^\eps - u^\eps )_t + \big( (u_\eps^2 - u_\eps^3) \log u^\eps - u^\eps + u_\eps^2 \big)_x 
= - D^\eps + \eps \, C^\eps_x 
$$ 
with $D^\eps := \eps \, u_\eps^3 \, |u^\eps_x|^2 + \gamma(\eps) \, | (u_\eps^2 \, u^\eps_x)_x|^2 \geq 0$. 
In the limit $\eps \to 0$ 
one deduces the single {\rm entropy inequality associated with the thin liquid film model}
$$
( u \log u - u )_t  + \big( (u^2 - u^3) \log u - u + u^2 \big)_x \leq 0.
$$
\end{lemma}


\subsection{Generalized Camassa-Holm model}
 
Consider now the following conservation law 
\be
\label{Cama}
u^\eps_t + f(u^\eps)_x = \eps \, u_{xx} + \alpha \, \eps^2 \,
 ( u^\eps_{txx} + 2 \, u^\eps_x \, u^\eps_{xx} + u^\eps \, u^\eps_{xxx}),  
\ee
in which $\alpha$ is a fixed parameter and $\eps \to 0$. This equation arises as a simplified model for 
shallow water when wave breaking takes place, and  
was studied by Bressan and Constantin \cite{BressanConstantin}, 
Coclite and Karlsen \cite{CocliteKarlsen},  
and LeFloch and Mohamadian \cite{LeFlochMohamadian}.

\begin{lemma}[Generalized Camassa-Holm model]
 For the augmented model \eqref{Cama} one has 
$$
\aligned
\bigl( (|u^\eps|^2 + \alpha \eps^2 \, |u_x^\eps|^2)/2 \bigr){_t} + {F(u^\eps)_x} 
= - \eps \, |u^\eps_x|^2 + \eps \, C^\eps_x,   
\endaligned
$$
which in the limit $\eps \to 0$ implies the following {\rm entropy inequality associated with the 
generalized Camassa-Holm model}
$$
\bigl( u^2/2 \bigr)_t + F(u)_x \leq 0, \qquad F' := u \, f'. 
$$
\end{lemma}

The entropy inequality coincides with the one in Lemma~\ref{DDD} for the diffusion-dispersion model.  
As pointed out in \cite{LeFlochMohamadian}, 
limiting solutions look quite similar but still do not coincide with the ones obtained with 
the diffusion-dispersion model.

 
\subsection{Phase transition model}
\label{PT}

Let us return to the model \eqref{555} cited in the introduction which, without regularization, reads
\be
\label{5550}
\aligned 
 w_t - v_x &  = 0,
\\
v_t - \sigma(w)_x 
&  = 0.  
\endaligned    
\ee
For many typical elastic materials, we have  
\be
\label{hypp}
\sigma'(w) >0 \quad \text{ for all } w>-1  
\ee
so that \eqref{5550} is strictly hyperbolic with two distinct wave speeds, $-\lam_1 = \lam_2 = c(w)$
(the sound speed). 
The two characteristic fields are genuinely nonlinear if and only if 
$\sigma''$ never vanishes. 
However, many materials encountered in applications do not satisfy this condition, 
 but rather loose convexity at $w=0$, that is,
$$
\sigma''(w) \gtrless 0 \quad \text{ if } \quad  w \gtrless 0. 
$$
One mathematical entropy pair of particular interest is the one associated with 
the total energy of the system and given by 
\be
\label{eee}
\aligned
U(v,w) =& \frac {v^2} 2 + \Sigma(w), 
\qquad F(v,w) = - \sigma(w) \, v, 
\\
\Sigma(w):=& \int_0^w \sigma(s) \, ds. 
\endaligned
\ee
The entropy $U$ is {\it strictly convex} under the assumption \eqref{hypp}.

Material undergoing phase transitions may be described by the model \eqref{5550} but with 
a {\sl non-monotone} stress-strain function satisfying  
\be
\label{888} 
\aligned 
& \sigma'(w) >0, \quad w \in (-1, w^m) \cup (w^M, +\infty), 
\\
& \sigma'(w) <0, \quad w \in (w^m, w^M) 
\endaligned    
\ee
for some constants $w^m < w^M$.
In the so-called {\sl unstable phase} $(w^m, w^M)$ 
the system admits two complex conjugate eigenvalues and is elliptic in nature.  
However, the solutions of interest from the standpoint of 
the hyperbolic theory lie {\sl outside} the unstable region. 
The system is hyperbolic in the non-connected set 
$\Ucal := \bigl(\RR \times (-1, w^m)\bigr) \cup 
      \bigl(\RR \times (w^M, +\infty)\bigr)$.   
One important difference with the hyperbolic regime 
is about the total mechanical energy \eqref{eee}, which
is still convex in each hyperbolic region, but (any extension) is not globally convex in (the convex closure of) $\Ucal$. 
  
The system \eqref{5550}-\eqref{888} and its augmented version \eqref{555} leads to complex wave dynamics, including 
hysteresis behavior, and is relevant to describe phase transitions in many different applications involving
solid-solid interfaces or fluid-gas mixtures. 

\begin{lemma}[Phase transition model]
\label{3331}
For the augmented model \eqref{555}, one has   
$$
\aligned 
& \left(\frac{v^2}{2} +  \Sigma(w) + {\alpha \, \eps^2 \over 2} \, w_x^2\right)_t 
-  \bigl(v \, \sigma(w) \bigr)_x
\\
&= 
\eps \, \bigl(v \, v_x \bigr)_x  - \eps \, v_x^2 + \alpha \, \eps^2 \, 
\bigl(v_x \, w_x - v \, w_{xx}\bigr)_x,
\endaligned 
$$ 
so that in the limit one formally obtains 
 the following {\rm entropy inequality associated with the phase transition model} 
\be
\label{299}
\Big(\frac{v^2}{2} + \Sigma(w) \Big)_t - \bigl(v \, \sigma(w)\bigr)_x \leq 0.  
\ee
\end{lemma} 

 
\subsection{Nonlinear phase transition model}
 
More generally, assume now an {\sl internal energy} function 
$e = e(w, w_x)$, and 
let us derive the field equations from the action
$$
J(y) 
: = \int_0^T \int_\Omega \Big(e(w, w_x) 
               - {v^2 \over 2} \Big) \, dx dt. 
$$
Precisely, considering the unknown functions $v$ and $w$ 
and defining the {\sl total stress} as  
$$
\Sigma(w,w_x, w_{xx}) 
:= { \del e \over \del w }(w,w_x) - \Big( { \del e \over \del w_x }(w,w_x) \Big)_x,   
$$ 
we can obtain 
$$ 
\aligned 
& v_t - \Sigma(w,w_x, w_{xx})_x = 0, \\ 
& w_t  - v_x = 0. 
\endaligned 
$$ 
Including next a nonlinear viscosity $\mu=\mu(w)$, 
we arrive at the nonlinear phase transition model which includes viscosity and capillarity 
effects: 
$$ 
\aligned 
& w_t - v_x = 0,  
\\
& v_t - \Sigma(w,w_x, w_{xx})_x = \bigl(\mu(w) \, v_x \bigr)_x.
\endaligned  
$$  

Again, the total energy 
$E(w,v,w_x) := e(w,w_x) + v^2 /2$ 
plays the role of a mathematical entropy, and we find
$$
E(w,v,w_x)_t - \bigl(\Sigma(w,w_x, w_{xx}) \, v\bigr)_x 
=  \Big( v_x \, { \del e \over \del w_x }(w,w_x) \Big)_x 
  + \bigl( \mu(w) \, v \, v_x \bigr)_x - \mu(w) \, v_x^2,
$$  
and once more, a single entropy inequality is obtained.

We now specialize this discussion with the important case that $e$ is quadratic in $w_x$. 
(Linear term should not appear because of the natural invariance 
of the energy via the transformation $x \mapsto -x$.)  
Setting, for some positive capillarity coefficient $\lam(w)$, 
$$
e(w,w_x) = \eps(w) + \lam(w) \, {w_x^2 \over 2}, 
$$ 
the total stress decomposes as follows: 
$$ 
\Sigma(w,w_x, w_{xx}) =  \sigma(w) + \lam'(w) \, {w_x^2 \over 2} - (\lam(w) \, w_x )_x, 
\quad 
\sigma(w) = \eps' (w),
$$ 
and the field equations take the form  
\be
\label{666}
\aligned 
& w_t  - v_x = 0,
\\
& v_t - \sigma(w)_x  
 = \Big( \lam'(w) \, {w_x^2 \over 2} - \bigl(\lam(w) \, w_x \bigr)_x \Big)_x 
   +  \bigl(\mu(w) \, v_x \bigr)_x. 
\endaligned 
\ee

\begin{lemma}[Nonlinear phase transition model]
For the augmented model \eqref{666}, one finds
$$
\aligned 
& \Big(\eps(w) + {v^2 \over 2} + \lam(w) \, {w_x^2 \over 2}\Big)_t 
  - \bigl(\sigma(w) \, v \bigr)_x 
\\
& =        \bigl( \mu(w) \, v \, v_x \bigr)_x  - \mu(w) \, v_x^2
		 + \Big(v \, {\lam'(w)\over 2} \, w_x^2  
                - v \, \bigl(\lam(w) \, w_x \bigr)_x 
                + v_x \, \lam(w)\, w_x \Big)_x,
\endaligned 
$$  
which leads to the same {\rm entropy inequality for the nonlinear phase transition model}
as \eqref{299} in Lemma~\ref{3331}. 
\end{lemma}

When the viscosity and capillarity are taken to be constants, we recover the example in Section~\ref{PT} above.  
The entropy inequality is identical for both regularizations.


\subsection{Magnetohydrodynamic model} 

Consider next the following simplified version of the equations of ideal magnetohydrodynamics 
\be
\label{834}
\begin{aligned}
{v_t + \big( (v^2 + w^2) \, v\big)_x} & = \eps \, v_{xx} + \alpha \, \epsilon \, w_{xx}, 
\\
{w_t + \big( (v^2 + w^2) \,Êw \big)_x} & = \eps \, w_{xx} - \alpha \, \epsilon \, v_{xx}, 
\end{aligned}
\ee
where $v,w$ denote the transverse components of the magnetic field, 
$\eps$ the magnetic resistivity, and $\alpha$ the so-called Hall parameter. 
The Hall effect taken into account in this model 
is relevant to investigate, for instance, the Earth's solar wind.

When $\alpha=0$, \eqref{834} was studied by Brio and Hunter \cite{BrioHunter},
Freist\"uhler et al. \cite{Frei,FP1,FP2}, Panov \cite{Panov-systems}, and others. 
The equations are not strictly hyperbolic, and for certain initial data 
the Riemann problem may admit up to two solutions.

When $\alpha \neq 0$, we refer to LeFloch and Mishra \cite{LeFlochMishra} who
demonstrated numerically that a kinetic function can be associated to this model. For the purpose 
of the present section, we observe the following. 

\begin{lemma}[Magnetohydrodynamic model]
For the augmented model \eqref{834} one has 
$$
\aligned
{(1/2) \, \big(v_\eps^2 + w_\eps^2\big)_t + (3/4) \, \big((v_\eps^2 + w_\eps^2)^2\big)_x 
=}& - \eps \, \big( (v^\eps_x)^2 + (w^\eps_x)^2 \big) + \eps \, C^\eps_x,  
\endaligned
$$ 
so that in the limit $\eps \to 0$ the following {\rm entropy inequality associated with 
 the magnetohydrodynamic model}
holds 
$$
(1/2) \, \big(v^2 + w^2\big)_t + (3/4) \, \big((v^2 + w^2)^2\big)_x \leq 0. 
$$
\end{lemma}


\subsection{Other physical models}
\label{OM}

The Buckley-Leverett equation for two-phase flows in porous media provides another example, studied 
in Hayes and Shearer \cite{HayesShearer} and Van~Duijn, Peletier, and Pop \cite{DuijnPeletierPop}. 
Finally, 
we list here several other models of physical interest which, however, have not yet received 
as much attention as the models presented so far. Since the hyperbolic flux part of these models
admits an inflection point and, moreover, physical modeling includes  
dispersive-type terms, 
it is expected that compressive and undercompressive shocks occur in weak solutions, at least in certain regimes
of applications. 
The actual occurrence of nonclassical shocks depends upon the form of the regularization,  
and further investigations are necessary about
the quantum hydrodynamics models \cite{Marcati, Jerome}, 
phase field models \cite{Caginalp, Ratz}, Suliciu-type models \cite{Carboux,Frid,Suliciu}, 
non-local models involving fractional integrals \cite{KLR,Rohde1,Rohde2},   
and discrete molecular models based, for instance,
on potentials of the Lennard-Jones type \cite{Bohm,DreyerH, NT, TV, Weinan}.


\section{Nonclassical Riemann solver} 

\subsection{Consequences of a single entropy inequality} 

For simplicity in the presentation, we consider a scalar equation  
\be
\label{127}
u_t + f(u)_x = 0, \qquad u=u(t,x)
\ee
with concave-convex flux
$$
\aligned 
& u \, f''(u) > 0 \quad  \text{for} \ u\neq 0, 
\\
& f'''(0) \neq 0, 
\qquad
\lim_{u \to \pm \infty} f'(u) = + \infty. 
\endaligned  
$$ 
To the flux $f$, we associate the {\sl tangent function} $\vn: \RR \to \RR$ 
(and its inverse denoted by $\varphi^{-\natural}$) defined by 
$$
f'(\vn(u)) = {f(u) - f(\vn(u)) \over u -\vn(u)}, \qquad  u \neq 0. 
$$

Weak solutions  $u \in L^\infty$ to \eqref{127}, by definition, satisfy 
$$
\iint \Big( u \, \varphi_t + f(u) \, \varphi_x \big) \, dxdt = 0 
$$
for every smooth, compactly supported function $\varphi$. 
If $u$ is a function with locally bounded variation, 
then $u_t$ and $u_x$ are locally bounded measures and \eqref{127} holds as an equality between measures. 
A shock wave $(u_-,u_+) \in \RR^2$, given by  
$$
u(t,x) = \begin{cases} 
u_-, \quad & x < \lambda \, t, 
\\
u_+, \quad & x > \lambda \, t,
\end{cases}
$$
 is a weak solution to \eqref{127} provided the Rankine-Hugoniot relation  
$$
- \lambda \, (u_+ - u_-) + f(u_+) - f(u_-) = 0
$$
holds. 
In the scalar case, this relation determines the shock speed uniquely: 
$$
\lambda = {f(u_-) - f(u_+) \over u_- -u_+} = :\ab(u_-,u_+). 
$$

Motivated by the physical models studied in the previous section, we
 impose that weak solutions satisfy a {\sl single entropy inequality}  
$$
U(u)_t + F(u)_x \leq 0, 
\qquad \qquad  U'' >0, \qquad F'(u) := f'(u) \, U'(u)
$$
for {\sl one} prescribed entropy pair. In other words, on the discontinuity $(u_-, u_+)$ we impose 
$$
\aligned
E(u_-,u_+) 
& := - {f(u_-) - f(u_+) \over u_- -u_+} \, \big( U(u_+) - U(u_-) \big) + F(u_+) - F(u_-)
\\
& \leq 0. 
\endaligned 
$$ 
It is easily checked that  
$$ 
E(u_-, u_+) = -\int_{u_-}^{u_+} U''(v) \, (v - u_-) \, 
    \Bigg({f(v) - f(u_-) \over v - u_-} - {f(u_+) - f(u_-) \over u_+ - u_-} 
    \Bigg) \, dv. 
$$

Note in passing that 
imposing all of the entropy inequalities
would lead to the {\sl Oleinik's entropy inequalities} 
\be
\label{129}
{f(v) - f(u_+) \over v - u_+} \leq {f(u_+) - f(u_-) \over u_+ - u_-} 
\ee
for all $v$ between $u_-$ and $u_+$. Imposing a single entropy inequality is much weaker than imposing \eqref{129}.

\begin{proposition}[Zero entropy dissipation function] 
Given a concave flux $f$ and a strictly convex entropy $U$, there exists a function
$\vfz:\RR\mapsto\RR$ such that 
\be
\label{ZER}
\aligned 
& E\big( u,\vfz(u) \big) = 0,           && \vfz(u)\neq u \quad (\text{ when } u\neq 0)
\\
& (\vfz \circ \vfz)(u) = u.
\endaligned 
\ee
Moreover, for instance when $u>0$, the entropy dissipation $E\big( u,\vfz(u) \big)$ is negative if and only if $u \in \big( \vfz(u), u \big)$. 
\end{proposition}

The above result follows from the identity (for $u_- \neq u_+$): 
$$
\aligned 
& \del_{u_+} E(u_-,u_+) = b(u_-,u_+) \, \del_{u_+} \ab(u_-,u_+), 
\\ 
& b(u_-,u_+) := U(u_-) - U(u_+) - U'(u_+) \, (u_-- u_+) >0, 
\endaligned 
$$
where the sign of the factor 
$$
\del_{u_+} \ab (u_-,u_+) = {f'(u_+) - \ab(u_-,u_+) \over u_+ - u_-}
$$
is easily determined in view of the concave-convex shape of $f$. 
 

\subsection{Admissible waves} 

We are in a position to deal with the Riemann problem, corresponding to the initial data 
$$
u(x,0) = \begin{cases}
u_l , & x < 0, 
\\
u_r, & x > 0, 
\end{cases}
$$ 
with $(u_l, u_r )\in \RR^2$. It turns out that a single entropy inequality allows for three types of waves
$(u_-, u_+)$: 
\begin{itemize}

\item {\sl Classical compressive shocks,} having 
$$
u_- >0, \qquad \vn(u_-) \leq u_+ \leq u_-, 
$$ 
which satisfy {\sl Lax shock inequalities}
$$ 
f'(u_-) \geq {f(u_+) - f(u_-) \over u_+ - u_-} \geq f'(u_+).
$$
Compressive shocks arise from smooth initial data: for instance, using the method of characteristics
one can write 
$$
u(t,x) = u\big(0, x- t \, f'(u(t,x)) \big),  
$$
and one sees that the implicit function theorem may  fail to determine the value $u(t,x)$ uniquely, 
whenever $t$ is sufficiently large. 
Compressive shocks arise also from singular limits,
 for instance from vanishing viscosity limits.

\item {\sl Nonclassical undercompressive shocks,} having 
$$
u_- >0, \qquad \vfz(u_-) \leq u_+ \leq \vn(u_-), 
$$  
for which all characteristics pass through it: 
$$
\min \Big( f'(u_-), f'(u_+) \Big) \geq {f(u_+) - f(u_-) \over u_+ - u_-}.
$$
The cord connecting $u_-$ to $u_+$ intersects the graph of $f$.  
Undercompressive shocks arise from certain (dispersive) singular limits, only. 
\eqref{cons21}.

\item {\sl Rarefaction waves,}
 which are Lipschitz continuous solutions $u$ depending only upon $\xi:=x/t$
and satisfy the ordinary differential equation  
$$
- \xi \, u(\xi)_\xi + f(u(\xi))_\xi = 0. 
$$
Precisely, a rarefaction consists of two constant states separated by a self-similar solution: 
$$
u(t,x) 
=
\begin{cases}
u_-,              &   x < t \, f'(u_-), 
\\
(f')^{-1} (x/t),  & t \, f'(u_-)  < x < t \, f'(u_-), 
\\
u_+,        & x > t \, f'(u_+). 
\end{cases}
$$
This construction makes sense 
provided 
$f'(u_-) < f'(u_+)$
and $f'$ is strictly monotone on the interval limited by $u_-$ and $u_+$.

\end{itemize}

By attempting to build a solution to the Riemann problem that uses only the above 
admissible waves, one realizes 
that the Riemann problem may admit a one-parameter family of solutions
satisfying a single entropy inequality.
Indeed, within an open range of initial data, one can combine together an arbitrary nonclassical shock plus a classical 
shock. 


\subsection{Entropy-compatible kinetic functions} 
 
At this stage of the discussion we introduce an additional admissibility requirement. 

\begin{definition}
A  {\rm kinetic function}
is a monotone decreasing, Lipschitz continuous function $\vf:\RR\mapsto\RR$ satisfying 
\be
\label{789}
\vfz(u) <  \vf(u) \leq \vn(u),  \quad u > 0.
\ee
The {\rm kinetic relation} 
$$
u_+ = \vf(u_-)
$$
then singles out one nonclassical shock for each left-hand state $u_-$. 
\end{definition}

Equivalently, one could prescribe the entropy dissipation rate across undercompressive shocks. 
The following two extremal choices could be considered:

\begin{itemize}

\item In the case $\vf = \vn$, the kinetic relation selects classical entropy solutions only, 
which in fact satisfy {\sl all} convex entropy inequalities. 

\item The choice $\vf = \vfz$ is not quite allowed in \eqref{789}
and would correspond to selecting {\sl dissipation-free shocks,} satisfying an {\sl entropy equality.}  

\end{itemize}

The property $ (\vfz \circ \vfz)(u) = u$ (see \eqref{ZER}) implies the {\sl contraction property}  
$$
|\vf\big( \vf(u)\bigr)| < |u|,   \quad u \neq 0. 
$$
which turns out to prevent oscillations with arbitrary large frequencies in solutions to the initial 
value problem. Introduce the companion threshold function $\vs:\RR \to \RR$
associated with $\phif$, defined by 
$$
{f(u) - f\big( \phif(u) \big) \over u - \phif(u)} = {f(u) - f\big( \vs(u) \big) \over u -\vs(u)}, \qquad  u \neq 0. 
$$

\begin{definition} 
The {\rm nonclassical Riemann solver} associates to any Riemann data $u_l, u_r$
the following entropy solution (for $u_l >0$, say): 

\begin{itemize}

\item a rarefaction wave if $u_r \geq u_l$, 

\item a classical shock if $u_r \in \big[ \vs(u_l), u_l \big)$,

\item if $u_r \in \big( \vf(u_l), \vs(u_l) \big)$, a nonclassical shock $\big( u_l, \vf(u_l)\big)$ 
followed by a classical shock $\big( \vf(u_l), u_r\big)$, and

\item if $u_r \leq \vf(u_l)$, a nonclassical shock $\big( u_l, \vf(u_l)\big)$
followed by a rarefaction wave $\big( \vf(u_l), u_r\big)$.

\end{itemize} 
\end{definition}

In conclusion, given a kinetic function $\vf$ compatible with an entropy, 
the Riemann problem admits a unique solution satisfying 
the hyperbolic conservation law, the Riemann initial data, 
the single entropy inequality, and a kinetic relation  $u_+ = \vf(u_-)$. 

Observe the $L^1$ continuous dependence property satisfied by any two entropy solutions $u,v$ 
associated with the same kinetic function, 
$$
\| u (t) - v(t) \|_{L^1(K)} \leq C(T, K) \, \| u (0) - v(0) \|_{L^1(K)}
$$
for all $t \in [0,T]$ and all compact set $K \subset \RR$. 
However, no pointwise version of this continuous dependence property holds, 
as the Riemann solution contain ``spikes'':
 some intermediate states depend discontinuously upon the initial data.


\subsection{Generalization to systems}

Nonclassical Riemann solvers are also known for several systems of interest. 
The $2 \times 2$ isentropic Euler,  nonlinear elasticity, and phase transition systems
were studied extensively in the mathematical literature.  
The nonclassical Riemann solver was constructed by Shearer et al. \cite{SchulzeShearer,ShearerYang}
(for the cubic equation) and LeFloch and Thanh \cite{LeFlochThanh3,LeFlochThanh1,LeFlochThanh2}
(for general constitutive equations admitting one inflection point). 
Uniqueness is also known for the Riemann problem
(within the class of piecewise smooth solutions)
when this system is strictly hyperbolic.
However, in the hyperbolic-elliptic regime,
two solutions are still available after imposing the 
kinetic relation. About the qualitative properties of solutions, 
see also important contributions by Hattori \cite{Hattori1,Hattori2,Hattori3},
Mercier and Piccoli \cite{MP1,MP2}, and Corli and Tougeron \cite{CST3}. 
Partial results are also available for 
the $3 \times 3$ Euler equations for van der Waals fluids \cite{LeFlochThanh-Waals}. 
Finally, 
for the construction of the nonclassical Riemann solver to general 
strictly hyperbolic systems of $N \geq 1$ conservation laws, 
we refer to Hayes and LeFloch \cite{HayesLeFloch-systems}.


\section{Kinetic relations associated with traveling waves}
 
\subsection{Traveling wave problem}
 
It was explained in the previous section that kinetic functions characterize the dynamics of nonclassical shocks
and allow one to solve the Riemann problem. 
The actual derivation of a kinetic relation is an essential issue and, in the present section, 
we explain how to derive it 
effectively from an analysis of traveling wave solutions to a given augmented model.
 In some cases, the kinetic function is computable by (explicit or implicit) analytic formulas.

For simplicity in the presentation, we consider the augmented model 
\be
\label{augm}
u_t + f(u)_x = \alpha \, \bigl( |u_x|^p \, u_x\bigr)_x + u_{xxx} 
\ee
in which $f$ is assumed to be a concave-convex function while
$\alpha>0$ and $p \geq 0$ are prescribed parameters. 

Searching for solutions $u(t,x) = w(y)$ depending only on the variable $y = x - \lam \, t$, we arrive at 
the second-order ordinary differential equation  
\be
\label{augm1}
- \lam \, (w - u_-)  + f(w) - f(u_-) = \alpha \, |w'|^p \, w' + w'' 
\ee
with boundary conditions
$$
\lim_{y \to \pm \infty} w(y) = u_\pm. 
$$
Here, $u_\pm$ and $\lam$ are constant states satisfying the Rankine-Hugoniot relation.

\begin{remark} 1. The parameter $\eps$ has been removed by rescaling of the original augmented equation. 

2. All of the results in this section extend to the more general model 
$$
u_t + f(u)_x  = \alpha \, \bigl( b(u,u_x) \, |u_x|^p \, u_x \bigr)_x + \bigl(c_1(u) \, \big(c_2(u) \, u_x\big)_x\bigr)_x, 
$$  
where the functions $b(u,v), c_1(u), c_2(u)$ are continuous and positive, 
and $b(u, v)\, |v|^p v$ is monotone increasing in $v$.  
\end{remark}

In view of \eqref{augm1}, the fundamental questions of interest are the following ones: 
do there exist traveling wave solutions associated with classical and/or with 
nonclassical shock waves $(u_-, u_+)$~? 
Can one associate a kinetic function $\varphi^\flat_{\alpha,p}$ to this model~?
If so, is this kinetic function monotone ? What is the local behavior of $\vf$ at $u=0$~? 
How does $\varphi^\flat_{\alpha,p}$ depend upon the two parameters $\alpha,p$~?

Answers to these questions were obtained first for the cubic flux function, 
by deriving {\sl explicit} formulas for the kinetic function: cf.~for  
$p = 0$, Shearer et al. \cite{JMS} and, for $p = 1$, 
Hayes and LeFloch \cite{HayesLeFloch-scalar}. In the latter case, it was observed that 
$\vfpalpha_{\alpha,1}'(0) = {\varphi^\flat_0}'(0)= -1$. 
General flux-functions and general regularization were covered
 by Bedjaoui and LeFloch in the series of papers \cite{BedjaouiLeFloch1}--\cite{BedjaouiLeFloch5}.


\subsection{Existence and asymptotic properties of kinetic functions}

To show the existence of the kinetic function, we reformulate \eqref{augm1} as a first-order system in the plane
$(w,w')$. The left-hand state and the speed being fixed, 
the corresponding equilibria are the solutions $u_1 < u_2 < u_3 = u_-$ to 
$$
- \lam \, (w - u_-)  + f(w) - f(u_-) = 0,
$$
which admits two non-trivial solutions (beyond $u_-$). 
Equilibria may be saddle points (two real eigenvalues with opposite signs) or 
nodes (two eigenvalues with same sign). A phase plane analysis shows that
there exist saddle-node connections from $u_-$ to $u_2$ (corresponding to classical shocks) 
as well as saddle-saddle connections from $u_-$ to $u_1$ (corresponding to nonclassical shocks). 
To state the results precisely, we introduce the following set of all admissible shocks
$$
S(u_-) := \bigl\{ u_+ \, / \, \text{ there exists a TW connecting $u_\pm$ }\bigr\}
$$ 
The following theorem is established in Bedjaoui and LeFloch \cite{BedjaouiLeFloch4}, and shows that  
to the augmented model one can associate a unique kinetic function which, furthermore,
is monotone and 
satisfies all the assumptions required in the theory of the Riemann problem.

\begin{theorem}[Existence of the kinetic function]
\label{KKK}
For each $\alpha >0$ and $p \geq 0$, 
consider the traveling wave problem \eqref{augm1} for the augmented model \eqref{augm}. 
Then, there exists a kinetic function $\varphi^\flat_{\alpha,p}: \RR \to \RR$ which is 
locally Lipschitz continuous, 
strictly decreasing, and such that, for instance when $u >0$,  
$$
\aligned 
& S(u) = \bigl\{\varphi^\flat_{\alpha,p}(u)\bigr\} \cup \bigl(\varphi^\sharp_{\alpha,p}(u), u\bigr], 
\\
& \vfz(u) < \varphi^\flat_{\alpha,p}(u) \leq \vn(u), 
\endaligned 
$$
where $\varphi^\flat_{\alpha,p}$ is the companion function associated with $\varphi^\flat_{\alpha,p}$. 
Moreover, for $0 \leq p \leq 1/3$, 
there exists a threshold function $A^\natural_p$ satisfying 
$$ 
 {\anp: \RR \to [0, \infty)} \text{ Lipschitz continuous, } \quad \anp(0) = 0,  
$$
so that 
$$
 \vfalpha(u) = \vn(u) \quad \text{ if and only if } \quad \alpha \geq \anp(u), 
$$
implying that all shocks of sufficiently small strength are classical.
On the other hand, for for all $p> 1/3$, one has 
$$
\vfalpha(u) \neq \vn(u) \qquad (u \neq 0),
$$
implying that there exist nonclassical shocks of arbitrarily small strength. 
\end{theorem}

The following local behavior is relevant in 
the existence theory for the initial value problem.

\begin{theorem} [Asymptotic properties of the kinetic function] 
Under the assumptions and notation of Theorem~\ref{KKK}, 
the behavior of infinitesimally small shocks is described as follows: 
\begin{itemize}

\item $p=0$:
$$
\aligned 
& \vfpalpha'(0)=\vn'(0)= -1/2,
\\
& \anp(0)=0, \qquad  \anp'(0\pm) \neq 0; 
\endaligned
$$
\item $0< p \leq 1/3$:
$$
\aligned
& \vfpalpha'(0) = -1/2, 
\\
& \anp(0)=0, \quad \anp'(0\pm) = + \infty; 
\endaligned
$$
\item $1/3 < p <1/2$:
$$
\vfpalpha'(0) = -1/2;
$$
\item $p=1/2$: 
$$ 
\vfpalpha'(0) \in \big( {\varphi^{-\flat}_0}'(0), -1/2\big)= (-1, -1/2), 
$$  
$$
\lim_{\alpha \to 0+} \vfpalpha'(0) = -1, 
\qquad \lim_{\alpha \to +\infty}\vfpalpha'(0) = -1/2; 
$$ 
\item $p>1/2$: 
$$
\vfpalpha'(0) = -1. 
$$
\end{itemize} 
\end{theorem}

\begin{remark} Explicit formulas for the kinetic function  \cite{JMS,HayesLeFloch-scalar,BedjaouiLeFloch4}.
are available when the flux is a cubic, say $f(u)= u^3$ and $p=0, 1/2$ or $1$.  
In particular, when $p = 1/2$ the kinetic function turns out to be the {\sl linear} function  
$$
\vf(u) = - c_\alpha \, u, \qquad  c_\alpha \in (1/2,1). 
$$
\end{remark}


\subsection{Generalization to systems}

The existence and properties of traveling waves for 
the nonlinear elasticity and the Euler equations 
are known in both the hyperbolic \cite{SchulzeShearer,BedjaouiLeFloch2} 
and the hyperbolic-elliptic regimes \cite{Truskinovsky2,BedjaouiLeFloch3,Benzoni,ShearerYang}.
For all other models, only partial results on traveling waves are available. 
 
The existence of nonclassical traveling wave solutions for the thin liquid film model is proven 
by Bertozzi and Shearer in \cite{BertozziShearer}. 
For this model, no qualitative information on the properties of these traveling waves is known, 
and, in particular, the existence of the kinetic relation has not been rigorously established yet.
The kinetic function was computed numerically in LeFloch and Mohamadian \cite{LeFlochMohamadian}. 
For the $3\times 3$ Euler equations, we refer to \cite{BedjaouiLeFloch5}.

The Van de Waals model admits {\sl two} inflection points and leads 
to {\sl multiple} traveling wave solutions. Although the physical significance
of the ``second'' inflection point is questionable, 
given that this model 
is extensively used in the applications it is important to investigate whether additional features
arise. Indeed, it was proven by Bedjaoui, Chalons, Coquel, and LeFloch 
\cite{BCCL} that {\sl non-monotone nonclassical} traveling wave profiles exist, and that 
a single kinetic function is not sufficient to single out the physically relevant solutions.

The Van der Waals-type model with viscosity and capillarity included reads 
$$  
\aligned 
& \tau_t - u_x = 0, 
 \\
& u_t + p(\tau)_x = \alpha \big( \beta(\tau) \, |\tau_x|^q \, u_x \big)_x - \tau_{xxx},  
\endaligned
$$
where $\tau$ denotes the specific volume,
$u$ the velocity, and 
 $\alpha$ the viscosity/capillarity ratio, 
while $q \geq 0$ and $\beta >0$ are parameters.  
We assume here the following convex/concave/convex pressure law:  
$$
\aligned
& p''(\tau) \geq 0,  \quad \tau \in (0,a) \cup (c,+\infty)
\\
& p''(\tau) \leq 0,  \quad \tau \in ( a, c), 
\qquad  p'(a) >0. 
\endaligned
$$ 

To tackle the traveling wave analysis we consider for definiteness 
a $2$-wave issuing from $(\tau_0,u_0)$ at $-\infty$ with speed $\lam>0$, so  
$$
\aligned
& \lam \, (\tau - \tau_0)+ u - u_0 = 0, 
\\
&  \lam \, (u - u_0) - p(\tau) + p(\tau_0)
= -\alpha \, \beta(\tau) |\tau'|^q u'   + \tau''. 
\endaligned
$$
We perform a phase plane analysis in the plane $(\tau,\tau')$: the equations consist of a 
second-order differential equation plus an algebraic equation. 
Fix a left-hand state $\tau_0$ and a speed $\lam$ within the interval 
where there exist three other equilibria $\tau_1,\tau_2,\tau_3$. 
These data are constrained to satisfy the following entropy inequality 
$$
-\lam \, \widetilde U'  + \widetilde F'  = - \alpha \, \beta(\tau) |\tau'|^q (u')^2 < 0
$$
with entropy  
$$
\widetilde U := -\int^\tau p(s) \, ds + {u^2 \over 2} + {(\tau')^2 \over 2}
$$
and 
entropy flux
$$
\widetilde F := u \, p(\tau) +\lam\,  (\tau')^2 + u \, \tau'' 
         - u \, \alpha \, \beta(\tau) \, |\tau'|^q \, u'.
$$

\begin{lemma} [Classification of equilibria] Consider the Van der Waals-type model above. 
For all $q\geq 0$, the equilibria $(\tau_0, 0)$ and $(\tau_2, 0)$ are {saddle} points
(two real eigenvalues with opposite signs).
 For $q = 0$ and $i = 1, 3$, the point $(\tau_i, 0)$ is 
\begin{itemize}
\item a {stable node} (two negative eigenvalues)
if $p'(\tau_i) + \lam^2 \leq (\alpha \lam \beta(\tau_i))^2 / 4$, and 
\item a {stable spiral}  
(two eigenvalues with the same negative real part and with opposite sign and  non-zero imaginary parts)
if $p'(\tau_i) + \lam^2 > (\alpha \lam \beta(\tau_i))^2/ 4$. 
\end{itemize}
For  $q>0$ the equilibria $(\tau_1, 0)$ and $(\tau_3, 0)$ are {centers} 
(two purely imaginary eigenvalues). 
\end{lemma}

In \cite{BCCL}, it is proven that  
there exists a {decreasing sequence} of diffusion/dispersion ratio $\oalpha_n(\tau_0,\lam) \to 0$
for $n \geq 0$ such that: 
\begin{itemize}

\item For $\alpha = \oalpha_n$ there exists a nonclassical traveling 
wave with $n$ oscillations connecting $\tau_0$ to $\tau_2$. 

\item For $\alpha \in (\oalpha_{2 m + 2}, \oalpha_{2 m + 1}) \cup (\oalphaz, +\infty)$, there exists  
a classical traveling wave connecting $\tau_0$ to $\tau_1$.

\item For $\alpha \in (\oalpha_{2 m + 1}, \oalpha_{2 m})$ there exists a classical traveling wave
 connecting $\tau_0$ to $\tau_3$.

\end{itemize}

In comparison, in the case of a single inflection point one has 
a single critical value $\oalpha_0(\tau_0,\lam)$, only. Here, we have infinitely many 
non-monotone, nonclassical trajectories associated with a sequence $\alpha_n \to 0$.  
This new feature observed with the van der Waals model indicates that 
the right-hand side across a given undercompressive wave
is not unique, 
and several kinetic functions  
should be introduced, leading also to non-uniqueness for the Riemann problem.


\section{Existence and uniqueness theory for nonclassical entropy solutions} 

\subsection{Dafermos' front tracking scheme}  

For simplicity, consider a conservation law \eqref{127} with concave-convex flux 
and impose the initial data 
\be
\label{idata} 
u(x,0) = u_0(x)
\qquad u_0 \in BV(\RR), 
\ee
where $u_{0,x}$ is a bounded measure and the total variation 
$TV(u_0)$ represents the total mass of this measure. 

By considering the nonclassical Rieman  solver based on some kinetic function $\varphi^\flat$
and following Dafermos \cite{Dafermos-front},
we construct a piecewise constant approximation $u^h: \RR^+ \times \RR \to \RR$, as follows.
First, one 
approximates the initial data $u_0$ with a piecewise constant function $u^h(0, \cdot)$. 
At the time $t=0$, one then solves a Riemann problem at each jump point of $u^h(0, \cdot)$. 
If necessary, one replaces rarefaction waves by several small fronts, traveling with the 
Rankine-Hugoniot speed. 
Then, at each interaction of waves, solve a new Riemann problem and continue the procedure inductively
in order to construct a globally defined, piecewise constant approximate solution $u^h= u^h(t,x)$. 
 
Clearly, in order to prove the convergence of the above approximation method, several difficulties must be overcome.
First of all, one needs to show that the total number of wave fronts as well as the total number of interaction points 
remain finite for all fixed time. In the scalar case under consideration and
 for concave-convex flux-functions, 
this is actually an easy matter since each nonclassical Riemann solution contains at most two outgoing waves.  
 Most importantly, one needs to derive a uniform bound (independent of $h$) 
on the total variation $TV(u^h(t, \cdot))$.
However, due to the lack of monotonicity of the nonclassical Riemann solver,  
the (standard) total variation may increase at interactions. 
For systems, further difficulties arise due to the lack of regularity of the wave curves, 
and one must also control nonlinear interactions between waves of different characteristic families.

Our assumptions on the kinetic function are very mild. 
We require that 
 $\vf : \RR \to \RR$ is Lipschitz continuous, monotone decreasing,  
 and that the second iterate of $\vf$ is a {strict contraction}: for $K \in (0,1)$
\be
\label{hipo}
       \big| \vf \circ \vf(u) \big| \leq K |u|, \qquad   u \neq 0.
\ee
Since 
$\big| \vf \circ \vf(u) \big| < |u|$
 for $u \neq 0$, 
this is  equivalent to imposing  
$\Lip_{u=0} (\vf\circ \vf)  < 1$, 
that is, only a condition on nonclassical shocks with infinitesimally small strength. 
For systems of equations, similar conditions make sense and are realistic for the application.


\subsection{Generalized wave strength} 

Following Baiti, LeFloch, and Piccoli~\cite{BLP-contemp} and
Laforest and LeFloch~\cite{LaforestLeFloch}, we define  the following {\sl generalized wave strength} 
\be
\label{ss}
\sigma(u_-, u_+) := | \psi(u_-) - \psi(u_+) |, 
\qquad \qquad \quad 
\psi(u) :=
\begin{cases}
u,     & u>0, 
\\
\varphi_0^\flat(u),  & u < 0. 
\end{cases}
\ee
This definition has several advantages. First, it compares states with the same sign.  
Second, it is ``equivalent'' to the standard definition of strength, in the sense that 
$$
\underline C \, | u_- - u_+ | \leq \sigma (u_-, u_+) \leq \overline C  \, | u_- - u_+ |.
$$
Third, it enjoys a continuity property as $u_+$ crosses $\varphi^\sharp(u_-)$, 
during a transition from a single crossing shock to a two wave pattern:  
$$ 
\aligned
\sigma(u_-,\vs(u_-)) & = \big| u_- - \vfz \circ \vs( u_-) \big| 
\\
& = \big| u_- - \vfz \circ \vf( u_-) \big| + \big| \vfz \circ \vf( u_-) -  \vfz \circ \vs( u_-) \big|
\\
& = \sigma(u_-, \vf(u_-)) + \sigma( \vf(u_-), \vs(u_-)).    
\endaligned
$$

We then define a {\sl generalized total variation functional,} 
for a piecewise constant function $u=u(t, \cdot)$ made of shock or rarefaction fronts $(u_-^\alpha, u_+^\alpha)$, 
by 
$$
   V \big(  u(t) \big) := \sum_\alpha  \sigma(  u_-^\alpha, u_+^\alpha ), 
$$ 
which again is ``equivalent'' to the standard total variation
$$
TV \big(  u(t) \big) := 
           \sum_\alpha  \big|  u_-^\alpha - u_+^\alpha \big|.
$$

A classification of all possible wave interaction patterns is given in \cite{ABLP,LeFloch-book},  
and about $20$ cases must be distinguished. 
Certain cases give rise to an increase of the standard total variation, which is not even proportional 
to the (smallest) strength of the incoming waves.  For instance, a (decreasing) classical shock may interact with a 
(decreasing) rarefaction coming from the righ-hand side, 
and transform into a (decreasing) nonclassical shock followed by an (increasing!) classical shock. 
Here, the outgoing wave profile is non-monotone and 
the standard total variation $TV \big(  u^h(t) \big)$ increases.
However, $V \big(  u^h(t) \big)$ does decrease.
Another possible interaction is provided by a (decreasing) nonclassical shock 
which hits an (increasing) classical shock and transforms itself into a (decreasing) 
classical shock. In that case, significant decay of the total variation occurs, and 
both the standard total variation $TV \big(  u^h(t) \big)$ and the generalized one $V \big(  u^h(t) \big)$ 
decrease.

  
\subsection{Existence theory}

We have the following estimates and existence theory. 

\begin{proposition}[Diminishing generalized total variation property]  
Consider a kinetic function satisfying the strict contraction property \eqref{hipo}. 
Then,  along a sequence of front tracking approximations based on the nonclassical Riemann solver, 
the generalized total variation functional $V= V\big( u^h(t) \big)$ is {non-increasing}.
\end{proposition}

\begin{theorem} [Existence of nonclassical entropy solutions \cite{BLP-existence,LaforestLeFloch}] 
\label{the}  
Consider a kinetic function compatible with a convex entropy $U$
and satisfying the strict contraction property \eqref{hipo}. 
For each initial data $u_0 \in BV(\RR)$, the 
wave front tracking approximations $u^h=u^h(t,x)$
constructed from the nonclassical Riemann solver 
satisfy 
$$ 
\aligned
& \|u^h(t)\|_{L^\infty(\RR)} \lesssim \|u_0\|_{L^\infty(\RR)}, \qquad 
\\
&
TV\big( u^h(t) \big) \lesssim TV(u_0),   
\\
& 
\|u^h(t) - u^h(s) \|_{L^1(\RR)} \lesssim |t - s|,  
\endaligned 
$$ 
and converge in $L^1$ to a weak solution 
$$
u=u(t,x) \in \Lip \big( [0,+\infty), L^1(\RR) \big) \cap L^{\infty}\big([0,+\infty),\operatorname{BV}(\RR)\big)
$$
of the initial value problem, 
which satisfies the entropy inequality 
$$ 
U(u)_t +  F(u)_x \leq 0.
$$  
\end{theorem} 

\ 

The uniquenesss of nonclassical entropy solutions is established in Baiti, LeFloch, and Piccoli \cite{BLP-unique}
within the class of functions with tame variation.

We note that pre-compactness of the sequence of approximate solutions follows from Helly's compactness theorem.
The behavior near $u=0$ is important to prevent a blow-up of the total variation. 
In fact, the condition
 ${\varphi^\flat}'(0-) \, {\varphi^\flat}'(0+) <1$
is indeed satisfied by kinetic functions generated by the nonlinear diffusion-dispersion model 
$$
\alpha \, \bigl( b(u,u_x) \, |u_x|^p \, u_x \bigr)_x + \bigl(c_1(u) \, (c_2(u) \, u_x)_x\bigr)_x 
$$
{\sl provided} $p<1/2$. Counter-example of blow-up of the total variation are available
 if $\vf(0) = -1$. See \cite{BLP-existence}.

Some additional existence results are available. 
Perturbations of a given nonclassical wave are analyzed in 
LeFloch \cite{LeFloch-ARMA}, Corli and Sabl\'e-Tougeron \cite{CST2,CST2}, Colombo and Corli \cite{CC1,CC2,ColomboCorli04a}, 
Hattori \cite{Hattori2}, 
Laforest and LeFloch \cite{LaforestLeFloch2}.
For a version of Glimm's wave interaction potential adapted to nonclassical solutions, 
we refer to \cite{LaforestLeFloch}. 
The $L^1$ continuous dependence of nonclassical entropy solutions is still an open problem, and  
it would be interesting to generalize to nonclassical shocks 
the techniques developed by 
Bressan et al. \cite{Bressan-book}, LeFloch et al. \cite{HuLeFloch,GoatinLeFloch,LeFloch-book}, and 
Liu and Yang \cite{LiuYang}.

\begin{remark} An alternative strategy to establish Theorem~\ref{the} is developed
 in \cite{BLP-charact} 
which, however, applies only to scalar equations and  
requires stronger conditions on the kinetic function.  
The proposed technique of proof therein is a decomposition of the real line 
into intervals where the approximate solution is alternatively increasing/decreasing. 
It requires a suitable application of Fillipov-Dafermos's theory of generalized characteristics to 
track the maxima and minima of the approximate solution. 
\end{remark}


\section{Finite difference schemes with controled dissipation}

\subsection{The role of the equivalent equation}

For simplicity in the presentation, we consider the case of the scalar conservation law 
\be
\label{222}
u_t +  f(u)_x = \eps \, u_{xx} + \alpha \, \eps^2 \, u_{xxx}, 
\ee
and formulate the following question. 
Denoting by $u_\alpha$ the limit when $\eps \to 0$
and by $\varphi^\flat_\alpha$ the associated kinetic function, 
can we design a numerical scheme  converging precisely to the function $u_\alpha$ ? 

One immediate and positive answer is provided by Glimm-type schemes, for 
which theoretical convergence results have been described in earlier sections. 
This convergence was illustrated by numerical experiments performed 
with the Glimm scheme in \cite{ChalonsLeFloch-Glimm}.
Another successful strategy is based on the level set technique implemented in
 \cite{HRL,CDMR,MerkleRohde} for a nonlinear elasticity model, with trilinear law, in two spatial dimensions, 
exhibiting complex interfaces with needles attached to the boundary. 
In addition, 
methods combining differences and interface tracking were also developed \cite{ZHL,Chalons,BCLL}, 
which ensure that the interface is sharp and (almost) exactly propagated. 

In the present section, we focus on {\sl finite difference schemes} and follow
Hayes and LeFloch \cite{HayesLeFloch-schemes}. 
Let $u_\alpha^{\Delta x}$ be some numerical solution and 
$v_{\alpha}  := \lim_{{\Delta x} \to 0} u_\alpha^{\Delta x}$ be its limit.
We are assuming that, at least at the practical level, the scheme is converging in a strong sense
and does generate nonclassical shocks whose dynamics can be described by a kinetic function $\psi^\flat_\alpha$.

It was observed in \cite{HayesLeFloch-schemes} that 
$$
v^\alpha\neq u^\alpha,  \qquad \psi^\flat_\alpha \neq \varphi^\flat_\alpha, 
$$ 
even if the scheme is conservative, consistent, high-order accurate, etc. 
 The point is that in both the continuous model and the discrete scheme, small scale features are critical to the selection of shocks. The balance between diffusive and dispersive features determines which shocks are selected. 
Note in passing that in 
nonconservative systems, the competition takes place between the (hyperbolic) propagation part and the (viscous) regularization \cite{HouLeFloch}. 

These small scale features cannot be quite the same at the continuous and
 discrete levels, since a continuous dynamical system of ordinary differential equations cannot be exactly represented by a discrete dynamical system of finite difference equations. 
 Consequently, 
 finite difference schemes do not converge to the correct weak solution when small-scaled are the driving factor for the selection of shock waves.

It was proposed \cite{HayesLeFloch-schemes} that $\psi^\flat_\alpha$ should be an  
accurate approximation of $\varphi^\flat_\alpha$
and  {\sl schemes with controled dissipation} were developed in \cite{HayesLeFloch-schemes,LeFlochRohde,ChalonsLeFloch1,ChalonsLeFloch2,LMR,LeFlochMohamadian}, 
which 
rely on high-order accurate, discrete hyperbolic flux and 
high-order discretizations of the augmented terms (diffusion, dispersion). 
More precisely, it is here required that the {\sl equivalent equation} of the scheme 
coincide
 with the augmented physical model,
up to a sufficiently high order of accuracy. For instance, for 
\eqref{222}, we require that after Taylor expanding the coefficients of the numerical scheme
\be
\label{333}
u_t +  f(u)_x 
=  \Delta x \, u_{xx} + \alpha \, (\Delta x)^2 \, u_{xxx} + O(\Delta x)^p, 
\ee
for $p \geq 3$ at least.

It was conjectured by the author that as $p \to \infty$ the kinetic function $\psi^\flat_{\alpha,p}$ associated with a scheme having the equivalent equation 
\eqref{333}
converges to the exact kinetic function $\varphi_\alpha^\flat$, that is, 
$$
\lim_{p \to \infty} \psi^\flat_{\alpha,p} = \varphi_\alpha^\flat.
$$
Support for this conjecture was recently provided by LeFloch and Mohamadian \cite{LeFlochMohamadian}, 
who performed extensive numerical tests 
for several models including the generalized Camassa-Holm and Van der Waals ones. 


\subsection{The role of entropy conservative schemes}

In addition, the role of entropy conservative schemes was stressed in \cite{HayesLeFloch-schemes,LeFlochRohde}. 
Schemes have been built 
from higher order accurate, entropy conservative, discrete flux. Such schemes 
satisfy discrete versions of the physically relevant entropy inequality, 
hence preserve exactly (and globally in time) an approximate entropy balance  

The design of entropy conservative schemes is based on entropy variable. Consider 
a system of conservation laws \eqref{IN.system} endowed with an entropy pair $(U,F)$. 
Suppose that $U$ strictly convex or, more generally, 
$f(u)$ can be expressed as a function of $v$, and 
let $v(u)= \nabla U(u) \in \Vcal := \nabla U(\Ucal)$ be the entropy variable. 
Finally, set $f(u) = g(v)$, $F(u) = G(v)$, and $B(v) = Dg(v)$. 
It is easily checked that $B(v)$ is symmetric, since $Dg(v) = Df(u) D^2 U(u)^{-1}$. So, there exists  $\psi(v)$ such that 
$g=\nabla\psi$ and, in fact, 
$$
\psi(v) = v \cdot g(v) - G(v). 
$$

On a regular mesh $x_j = j \, h$ ($j=\ldots, -1,0,1, \ldots$), 
consider $(2p+1)$-point, conservative, semi-discrete schemes 
$$
{d \over dt} u_j = - {1 \over h} \, (g^*_{j+1/2} - g^*_{j-1/2} ), 
$$
where $u_j = u_j(t)$ represents an approximation of $u(x_j,t)$. The discrete flux 
$$
g^*_{j+1/2} = g^*(v_{j-p+1}, \cdots, v_{j+p}),       
\qquad  
v_j = \nabla U(u_j) 
$$
must be consistent with the exact flux $g$, i.e.   
$$
g^*(v, \ldots, v) = g(v).  
$$

Tadmor in \cite{Tadmor-old} introduced the notion of entropy conservative schemes
and focused on essentially three-point schemes,
for which $g^*(v_{-p+1}, \cdots, v_p) = g(v)$ when $v_0= v_1=v$, 
which implies second-order accuracy. 
The two-point numerical flux \cite{Tadmor-old} 
$$
g^*(v_0, v_1) = \int_0^1 g(v_0 + s \, (v_1 - v_0)) \, ds, 
\qquad v_0, v_1 \in \Ucal.
$$
yields an entropy conservative scheme, satisfying 
$$
{d \over dt} U(u_j) + {1 \over h} \, (G^*_{j+1/2} - G^*_{j-1/2} ) =0, 
$$
with  
$$
\aligned
G^*(v_0, v_1) = 
& {1 \over 2} \, (G(v_0) + G(v_1))  
+ {1 \over 2} \, (v_0 + v_1) \, g^*(v_0, v_1) 
\\
& - {1 \over 2} \, (v_0 \cdot g(v_0) + v_1 \cdot g(v_1)), \qquad v_0, v_1 \in \Vcal. 
\endaligned
$$
This scheme admits the following second-order accurate, (conservative) equivalent equation~\cite{LeFlochRohde} 
$$
u_t +  f(u)_x = {h^2 \over 6} \, \Big(-g(v)_{xx}  
+ {1 \over 2} \, v_x \cdot Dg(v)_x\Big)_x 
$$
with $v = \nabla U(u)$.  This is not sufficient for our purpose of tackling the diffusion-dispersion model.

High-order entropy conservative schemes were discovered by LeFloch and Rohde \cite{LeFlochRohde}. 
Later, generalizations to arbitrarily high order were found in \cite{LMR,Tadmor-acta,Tadmor-recent}.

\begin{theorem} [Third-order entropy conservative schemes] 
Consider the semi-disrete $(2p+1)$-point scheme
$$
{d \over dt}  u_j = - {1 \over h} \, (g^*_{j+1/2} - g^*_{j-1/2} ),\quad
g^*_{j+1/2} = g^*(v_{j-p+1}, \cdots, v_{j+p})
$$
with numerical flux $g^*$ defined as follows,    
from any symmetric $N \times N$ matrices $B^*(v_{-p+2}, \cdots,v_p)$, 
$$
\aligned 
& g^*(v_{-p+1}, \cdots, v_p) 
\\
& = \int_0^1 g\big(v_0 + s \, (v_1 - v_0)\big) \, ds   -  {1 \over 12 } \Big(&&(v_2-v_1) \cdot B^*(v_{-p+2}, \cdots, v_p) 
\\
& \qquad  && -  (v_0- v_{-1}) \cdot B^*(v_{-p+1}, \cdots,v_{p-1})\Big). 
\endaligned
$$ 
This scheme is entropy conservative for the entropy $U$, with entropy flux 
$$
\aligned  
& G^*(v_{-p+1}, \cdots, v_p) 
\\
& = {1  \over 2} (v_0 + v_1) \cdot g^*(v_{-p+1}, \cdots, v_p) 
 - {1 \over 2} \Big(\psi^*(v_{-p+2}, \cdots, v_p) 
        + \psi^*(v_{-p+1}, \cdots, v_{p-1})\Big). 
\endaligned
$$
and 
$$\aligned 
& \psi^*(v_{-p+2}, \cdots, v_p) 
\\
& = v_1 \cdot g(v_1) - G(v_1) 
+ {1 \over 12} (v_1 - v_0) \cdot B^*(v_{-p+2}, \cdots, v_p) (v_1- v_2).    
\endaligned
$$
When $p=2$ and $B^*(v,v,v) = B(v) \Big( = Dg(v)\Big)$, 
this five-point scheme is third-order, at least. 
\end{theorem}

Although schemes with controled dissipation based on an analysis of their equivalent equation
do not converge to the exact solution determined by a given augmented 
model, still they provide a large class of practically useful schemes 
and allow one to ensure that the numerical kinetic function {\sl approaches} the exact one.

It is now established numerically that kinetic functions exist and are monotone for 
a large class of physically relevant models including thin liquid films,  generalized Camassa-Holm, nonlinear phase transitions, van der Waals fluids (for small shocks), and magnetohydrodynamics.
 
Computing the kinetic function has been found to be very useful to investigate
the effects of the diffusion/dispersion ratio, regularization, order of accuracy of the schemes, 
 the efficiency of the schemes, 
 as well as to make comparisons between several physical models. 

In addition, it should be noted that kinetic functions may be associated with schemes \cite{HayesLeFloch-schemes}: 
The Beam-Warming scheme (for concave-convex flux) produces non-classical shocks, while
no such shocks are observed with the Lax-Wendroff scheme. 
All of this depends crucially on the sign of the numerical dispersion coefficient.

 As mentioned earlier, similar issues arise in dealing with the numerical 
 approximation of nonlinear hyperbolic systems in nonconservative form
 for which we refer to \cite{HouLeFloch,Castro--Pares} and the references therein.  
 

\section{Concluding remarks}

Let us conclude by mentioning a few more issues of interest. 

Vanishing diffusion-dispersion limits in the context of the initial value problem 
have been investigated by several approaches. 
Tartar's compensated compactness method was applied to treat one-dimensional 
scalar conservation laws, by Schonbek \cite{Schonbek}, Hayes and LeFloch \cite{HayesLeFloch-scalar}, and 
LeFloch and Natalini \cite{LeFlochNatalini}. The $2 \times 2$ system of nonlinear 
elasticity system was tackled in \cite{HayesLeFloch-systems}. 
Singular limits for the Camassa-Holm equation were analyzed by Coclite and Karlsen \cite{CocliteKarlsen}. 
In all these works, conditions are imposed on the diffusion and dispersion parameters which are 
mild enough to allow for nonclassical shocks in the limit.

Another strategy based on DiPerna's measure-valued solutions \cite{DiPerna} 
 applies to multidimensional conservation laws. Strong convergence results were established
by Correia and LeFloch \cite{CorreiaLeFloch1,CorreiaLeFloch2} and Kondo and LeFloch \cite{KondoLeFloch}. 
The generalization to discontinuous flux is provided in Holden, Karlsen, and Mitrovic \cite{HKM}. 
DiPerna's theorem requires all of the entropy inequalities, and therefore these results do not cover the regime
of parameters allowing for nonclassical shocks.

A third approach is based on Lions, Perthame, and Tadmor's kinetic formulation \cite{LPT}
and again applies to multidimensional conservation laws. 
It was first observed by Hwang and Tzavaras \cite{HwangTzavaras} 
that the kinetic formulation extends to singular limits to conservation laws when 
diffusive and dispersive parameters are kept in balance. 
See also Hwang \cite{Hwang1,Hwang2} and Kwon \cite{Kwon}. 
The kinetic formulation was recently extended to non-local regularizations \cite{KLR}.

In another direction, LeFloch and Shearer \cite{LeFlochShearer} introduced a {\sl nucleation criterion}
 and 
a {\sl Riemann solver with kinetic and nucleation} for scalar conservation laws. 
They could cope with problems in which the traveling wave analysis does not select a 
unique solution to the Riemann problem when, even after imposing a kinetic relation
for undercompressive shocks, 
 one is still left with a classical and a nonclassical Riemann solution. Earlier on,
Abeyaratne and Knowles \cite{AbeyaratneKnowles1}
had introduced a nucleation criterion in the context of the model of elastodynamics
with a trilinear equation of state. 

Roughly speaking, the nucleation criterion 
imposes that a ``sufficiently large'' initial jump always ``nucleates''. 
The qualitative properties of the Riemann solver with kinetic and nucleation
were investigated in \cite{LeFlochShearer}:
 prescribing the set of admissible waves does not uniquely determine the Riemann solution, 
 and
 instability phenomena such as ``splitting-merging'' wave structures take place.
The analysis was recently extended to hyperbolic systems of 
conservation laws \cite{LaforestLeFloch2} and 
was also further investigated numerically \cite{LevyShearer1}. 
It is expected that the nucleation criterion will be particulary relevant for 
higher-order regularizations such as the one in the thin liquid film model.  

In conclusion, a large class of scalar equations and $2 \times 2$ systems are now well-understood. 
The Riemann problem is uniquely determined by an entropy inequality and a kinetic relation, 
and the kinetic function can be determined by an analysis of traveling solutions. 
For these example, the Riemann solution depends continuously upon its initial data. 
However, for other models such as the thin film model or the hyperbolic-elliptic model of phase transitions, 
the existing theory has limitations, and this suggests challenging open problems.

Due to a lack of space (and time), still many other issues could not be treated in this review. 
Kinetic relations are relevant also for nonconservative hyperbolic systems \cite{BerthonCoquelLeFloch}, 
and should play an important role in the mathematical modeling of multi-fluid and turbulence models.  
For an extensive literature on the nonlinear stability of undercompressive  shock waves 
(including for multi-dimensional problems), 
we refer to works by T.-P. Liu, Metivier, Williams, Zumbrun, and others. 
See \cite{Azevedo,HowardZumbrun,Isaacson,LiuZumbrun1,LiuZumbrun2,RaoofiZumbrun,SchecterShearer} and the references therein.


\section*{Acknowledgments}  

I am particularly grateful to Constantine M. Dafermos, Tai-Ping Liu, and Michael Shearer
for scientific discussions on this research work
and for their encouragments to pursue it. 

This paper was written as part of  the international research program on Nonlinear Partial Differential Equations 
held at the Centre for Advanced Study of the Norwegian Academy of Science and Letters
during the Academic Year 2008Ð-09. I am very grateful to Helge Holden and 
Kenneth Karlsen for their invitation and hospitality and for the opportunity to 
give a short-course, originally entitled ``Small-scale dependent shock waves. Theory, approximation, and applications''. The author was also supported by a DFG-CNRS collaborative grant between France and Germany
on ``Micro-Macro Modeling and Simulation of Liquid-Vapor Flows'', 
as well as by 
 the Centre National de la Recherche Scientifique (CNRS) and 
the Agence Nationale de la Recherche (ANR) via the grant 06-2-134423. 



\begin{thebibliography}{10}

\small 

\bibitem{AbeyaratneKnowles1} R. Abeyaratne and J.K. Knowles, 
Kinetic relations and the propagation of phase boundaries in solids,
Arch. Rational Mech. Anal. 114 (1991), 119--154.

\bibitem{AbeyaratneKnowles2} R. Abeyaratne and J.K. Knowles,
Implications of viscosity and
strain-gradient effects for the kinetics of propagating phase boundaries in solids,
SIAM J. Appl. Math. 51 (1991), 1205--1221.

\bibitem{ABLP} D. Amadori, P. Baiti, P.G. LeFloch and B. Piccoli, 
Nonclassical shocks and the Cauchy problem for nonconvex conservation laws, 
J. Differential Equations 151 (1999), 345--372.

\bibitem{Marcati} P. Antonelli and M. Marcati,
On the finite energy weak solutions to a system in quantum fluid dynamics,
Comm. Math. Phys. 287 (2009), 657--686. 
  
\bibitem{Azevedo} A. Azevedo, D. Marchesin, B.J. Plohr, and K. Zumbrun, 
Non-uniqueness of solutions of Riemann problems 
caused by 2-cycles of shock waves, 
Proc. Fifth Internat. Conf. on Hyperbolic Problems:
theory, numerics, applications, 
J. Glimm, M.J. Graham, J.W. Grove, and B.J. Plohr, ed., 
World Scientific Editions, 1996, pp.~43--51. 

\bibitem{BLP-contemp} P. Baiti, P.G. LeFloch, and B. Piccoli, 
Nonclassical shocks and the Cauchy problem. General conservation laws, 
Contemporary Math. 238 (1999), 1--25.

\bibitem{BLP-charact} P. Baiti, P.G. LeFloch, and B. Piccoli, 
BV Stability via generalized characteristics for nonclassical solutions of conservation laws, 
EQUADIFF'99, Proc. Internat. Conf. Differ. Equ., Berlin,  
August 1999, B. Fiedler, K. Gr\"oger, and J.Sprekels, editors, World Sc. Publ., 
River Edge, NY, 2000, pp.~289--294. 

\bibitem{BLP-unique} P. Baiti, P.G. LeFloch, and B. Piccoli,
Uniqueness of classical and nonclassical solutions for nonlinear hyperbolic systems, 
J. Differential Equations 172 (2001), 59--82.  

\bibitem{BLP-existence} P. Baiti, P.G. LeFloch, and B. Piccoli, 
Existence theory for nonclassical entropy solutions: scalar conservation laws, 
Z. Angew. Math. Phys. 55 (2004),Ê  927--945. 

\bibitem{BCCL} N. Bedjaoui, C. Chalons, F. Coquel, and P.G. LeFloch, 
Non-monotonic traveling waves in van der Waals fluids, Anal. Appl. 3 (2005), 419--446. 

\bibitem{BedjaouiLeFloch1} N. Bedjaoui and P.G. LeFloch,
Diffusive-dispersive traveling waves and kinetic relations I.
Non-convex hyperbolic conservation laws,
J. Differential Equations 178 (2002), 574--607.

\bibitem{BedjaouiLeFloch2} N. Bedjaoui and P.G. LeFloch,
Diffusive-dispersive traveling waves and kinetic relations II.
A hyperbolic-elliptic model of phase transitions dynamics,
Proc. Royal Soc. Edinburgh 132A (2002), 545-565.

\bibitem{BedjaouiLeFloch3}  N. Bedjaoui and P.G. LeFloch,
Diffusive-dispersive traveling waves and kinetic relations III.
An hyperbolic model from nonlinear elastodynamics,
Ann. Univ. Ferrara Sc. Mat. 47 (2002), 117--144.

\bibitem{BedjaouiLeFloch4}  N. Bedjaoui and P.G. LeFloch, 
Diffusive-dispersive traveling waves and kinetic relations. IV. Compressible Euler system, 
Chinese Ann. Appl. Math. 24 (2003), 17--34.

\bibitem{BedjaouiLeFloch5}  N. Bedjaoui and P.G. LeFloch, 
Diffusive-dispersive traveling waves and kinetic relations. V. Singular diffusion and dispersion terms, 
Proc. Royal Soc. Edinburgh 134A (2004), 815--844. 

\bibitem{Benzoni} S. Benzoni-Gavage,
Stability of subsonic planar phase boundaries in a van der Waals fluid,
Arch. Ration. Mech. Anal. 150 (1999), 23--55.

\bibitem{BHP96} E. Beretta, J. Hulshof, and L.A. Peletier, 
On an ordinary differential equation from forced coating flow, 
J. Differential Equations 130 (1996), 247--265.  

\bibitem{BerthonCoquel} C. Berthon and F. Coquel, 
Nonlinear projection methods for multi-entropies Navier-Stokes systems,
Math. of Comp. 76 (2007), 1163--1194. 

\bibitem{BerthonCoquelLeFloch} C. Berthon, F. Coquel, and P.G. LeFloch,   
Kinetic relations for nonconservative hyperbolic systems and applications, 
unpublished notes (2002).

\bibitem{BertozziMunchShearer} A. Bertozzi, A. M\"unch, and M. Shearer, 
Undercompressive shocks in thin film flow, 
Phys. D 134 (1999), 431-464.

\bibitem{BMSZ} A. Bertozzi, A. M\"unch, M. Shearer, and K. Zumbrun, 
Stability of compressive and undercompressive thin film traveling waves, 
European J. Appl. Math. 12 (2001), 253--291.  

\bibitem{BertozziShearer} A. Bertozzi and M. Shearer, 
Existence of undercompressive traveling waves in thin film equations, 
SIAM J. Math. Anal. 32 (2000), 194--213.
 
\bibitem{Bohm} T. B\"ohme, W. Dreyer, and W.H. M\"uller, 
Determination of stiffness and higher gradient coefficients by means of the embedded-atom method,
Contin. Mech. Thermodyn. 18 (2007), 411--441.
 
\bibitem{BCLL} B. Boutin, C. Chalons, F. Lagouti\`ere, and P.G. LeFloch, 
Convergent and conservative schemes for nonclassical solutions based on kinetic relations, 
Interfaces and Free Boundaries 10 (2008), 399--421. 

\bibitem{BianchiniBressan} S. Bianchini and A. Bressan,
Vanishing viscosity solutions of nonlinear hyperbolic systems,
Ann. of Math. 161 (2005), 223--342. 

\bibitem{Bouchut} F. Bouchut, C. Klingenberg, and K. Waagan, 
A multiwave approximate Riemann solver for ideal MHD based on relaxation. I. Theoretical framework,
Numer. Math. 108 (2007), 7--42. 

\bibitem{Bressan-book} A. Bressan,  
{\em Hyperbolic systems of conservation laws. The one-dimensional Cauchy problem,} 
Oxford Lecture Series Math. Appl., 20, Oxford University Press, Oxford, 2000. 

\bibitem{BressanConstantin} A. Bressan and A. Constantin, 
Global dissipative solutions of the Camassa-Holm equation,
Anal. Appl. 5 (2007), 1--27.  

\bibitem{BressanLeFloch} A. Bressan and P.G. LeFloch, 
Uniqueness of entropy solutions for systems of conservation laws, 
Arch. Rational Mech. Anal. 140 (1999), 301--331. 

\bibitem{BrioHunter} M. Brio and J.K. Hunter,
Rotationally invariant hyperbolic waves,
Comm. Pure Appl. Math. 43 (1990), 1037--1053. 

\bibitem{Caginalp} G. Caginalp and X. Chen, 
Convergence of the phase field model to its sharp interface limits,
European J. Appl. Math. 9 (1998), 417--445. 
 
\bibitem{Carboux} G. Carboux and B. Hanouzet, 
Relaxation approximation of some initial-boundary value problem for $p$-systems,
Commun. Math. Sci. 5 (2007), 187--203.

\bibitem{Castro--Pares} M.J. Castro, P.G. LeFloch, M.L. Munoz-Ruiz, and C. Pares,  
Why many theories of shock waves are necessary. Convergence error in formally path-consistent schemes, 
J. Comput. Phys. 227 (2008), 8107--8129. 

\bibitem{Chalons} C. Chalons,  
Transport-equilibrium schemes for computing nonclassical shocks. Scalar conservation laws,
Numer. Methods Partial Differential Equations 24 (2008), 1127--1147. 

\bibitem{ChalonsCoquel} C. Chalons and F. Coquel,
Numerical capture of shock solutions of nonconservative hyperbolic systems via kinetic functions. 
Analysis and simulation of fluid dynamics,
Adv. Math. Fluid Mech., pp.~45--68, Birkh\"auser, B\"asel, 2007. 

\bibitem{ChalonsLeFloch1} C. Chalons and P.G. LeFloch, 
High-order entropy conservative schemes and kinetic relations for van der Waals fluids, 
J. Comput. Phys. 167 (2001), 1--23.

\bibitem{ChalonsLeFloch2}  C. Chalons and P.G. LeFloch, 
A fully discrete scheme for diffusive-dispersive conservation laws, 
Numerische Math. 89 (2001), 493--509.

\bibitem{ChalonsLeFloch-Glimm} C. Chalons and P.G. LeFloch, 
Computing undercompressive waves with the random choice scheme : nonclassical shock waves, 
Interfaces and Free Boundaries 5 (2003), 129--158.

\bibitem{Charlotte} M. Charlotte and L. Truskinovsky, 
Towards multi-scale continuum elasticity theory,
 Contin. Mech. Thermodyn. 20 (2008), 133--161. 

\bibitem{CocliteKarlsen} G.M. Coclite and K.H. Karlsen,
A singular limit problem for conservation laws related to the Camassa-Holm shallow water equation,
Comm. Partial Differential Equations 31 (2006), 1253--1272. 

\bibitem{CC1} R.M. Colombo and A. Corli,
Continuous dependence in conservation laws with phase transitions,
SIAM J. Math. Anal. 31 (1999), 34--62. 

\bibitem{CC2} R.M. Colombo and A. Corli,
Stability of the Riemann semigroup with respect to the kinetic condition,
Quart. Appl. Math. 62 (2004), 541--551. 

\bibitem{ColomboCorli04a} R.M. Colombo and A. Corli,
Sonic and kinetic phase transitions with applications to Chapman-Jouguet deflagration, 
Math. Methods Appl. Sci. 27 (2004), 843--864. 

\bibitem{CDMR} F. Coquel, D. Diehl, C. Merkle, and C. Rohde, 
Sharp and diffuse interface methods for phase transition problems in liquid-vapour flows,
in 
 ``Numerical methods for hyperbolic and kinetic problems'', pp.~239--270, 
 IRMA Lect. Math. Theor. Phys., 7, Eur. Math. Soc., Z\"urich, 2005. 

\bibitem{CST} A. Corli and M. Sabl\'e-Tougeron, 
Kinetic stabilization of a nonlinear sonic phase boundary,
Arch. Rational Mech. Anal. 152 (2000), 1--63.

\bibitem{CST1} A. Corli and M. Sabl\'e-Tougeron, 
Stability of contact discontinuities under perturbations of bounded variation,
Rend. Sem. Mat. Univ. Padova 97 (1997), 35--60. 

\bibitem{CST2} A. Corli and M. Sabl\'e-Tougeron, 
Perturbations of bounded variation of a strong shock wave,
J. Differential Equations 138 (1997), 195--228.   

\bibitem{CST3} A. Corli and M. Sabl\'e-Tougeron, 
Kinetic stabilization of a nonlinear sonic phase boundary,
 Arch. Ration. Mech. Anal. 152 (2000), 1--63.  
 
\bibitem{CorreiaLeFloch1} J.M. Correia and P.G. LeFloch, 
Nonlinear diffusive-dispersive limits for multidimensional conservation laws, 
in ``Advances in Nonlinear P.D.E.'s and Related Areas", (Beijing, 1997), 
World Sci. Publ., River Edge, NJ, 1998, pp.~103--123. Available at: http://arxiv.org/abs/0810.1880.

\bibitem{CorreiaLeFloch2}
J.M. Correia and P.G. LeFloch, 
Nonlinear hyperbolic conservation laws, in 
``Nonlinear evolution equations and their applications'' (Macau, 1998),
World Sci. Publ., River Edge, NJ, 1999,  pp.~21--44.  

\bibitem{Dafermos-front} C.M. Dafermos, 
Polygonal approximations of solutions of the initial value problem for a conservation law,
J. Math. Anal. Appl. 38 (1972), 33--41.

\bibitem{Dafermos-entropy} C.M. Dafermos, 
Hyperbolic systems of conservation laws, 
Proceedings ``Systems of Nonlinear Partial Differential Equations'', 
J.M. Ball editor, NATO Adv. Sci. Series C, 111, Dordrecht D. Reidel 
(1983), 25--70.

\bibitem{Dafermos-book} C.M. Dafermos,
{\em Hyperbolic conservation laws in continuum physics,}  
Grundlehren Math. Wissenschaften Series 325, Springer Verlag, 2000.

\bibitem{DLM} G. Dal Maso, P.G. LeFloch, and F. Murat, 
Definition and weak stability of nonconservative products, 
J. Math. Pures Appl. 74 (1995), 483--548.

\bibitem{DOW} C. De Lellis, F. Otto, and M. Westdickenberg,
Minimal entropy conditions for Burgers equation,
Quart. Appl. Math. 62 (2004), 687--700. 

\bibitem{DiPerna} R.J. DiPerna, 
Measure-valued solutions to conservation laws, 
Arch. Rational Mech. Anal. 88 (1985), 223--270. 

\bibitem{DreyerH} W. Dreyer and M. Herrmann,
Numerical experiments on the modulation theory for the nonlinear atomic chain,
Phys. D 237 (2008), 255--282.
  
\bibitem{Dreyer} W. Dreyer and M. Kunik,
Maximum entropy principle revisited,
Contin. Mech. Thermodyn. 10 (1998), 331--347. 

\bibitem{Weinan} W. E and D. Li,
The Andersen thermostat in molecular dynamics,
Comm. Pure Appl. Math. 61 (2008), 96--136. 

\bibitem{FanLiu} H.T. Fan and H.-L.~Liu,
Pattern formation, wave propagation and stability in conservation laws
with slow diffusion and fast reaction, 
J. Differ. Hyper. Equa. 1 (2004), 605--636. 

\bibitem{FanSlemrod} H.-T. Fan and M. Slemrod,
The Riemann problem for systems of conservation laws
of mixed type, in ``Shock induced transitions and phase
structures in general media'', Workshop held in Minneapolis (USA), Oct. 1990,
Dunn J.E. (ed.) et al., IMA Vol. Math. Appl. 52 (1993), pp. 61-91. 

\bibitem{Frei} H. Freist\"uhler, 
Dynamical stability and vanishing viscosity: A case study 
of a non-strictly hyperbolic system of conservation laws, 
Comm. Pure Appl. Math. 45 (1992), 561--582. 


\bibitem{FP1} H. Freist\"uhler and E.B. Pitman,
 A numerical study of a rotationally degenerate hyperbolic system. I. The Riemann problem,
  J. Comput. Phys. 100 (1992), n306--321. 

\bibitem{FP2} H. Freist\"uhler and E.B. Pitman,
A numerical study of a rotationally degenerate hyperbolic system. II. The Cauchy problem,
 SIAM J. Numer. Anal. 32 (1995), 741--753. 

\bibitem{Frid} H. Frid and I-S. Liu,
Phase transitions and oscillation waves in an elastic bar,
in``Fourth Workshop on Partial Differential Equations'', Part I (Rio de Janeiro, 1995),
Mat. Contemp. 10 (1996), 123--135. 

\bibitem{GoatinLeFloch}
P. Goatin and P.G. LeFloch,
Sharp $L\sp 1$ continuous dependence of solutions of bounded variation for hyperbolic systems of conservation laws,
Arch. Ration. Mech. Anal. 157 (2001), 35--73. 
 
\bibitem{Grinfeld1} M. Grinfeld,
Nonisothermal dynamic phase transitions,
Quart. Appl. Math. 47 (1989), 71--84. 

\bibitem{Hattori1} H. Hattori, 
The Riemann problem for a van der Waals fluid with entropy rate admissibility criterion---isothermal case,
Arch. Rational Mech. Anal. 92 (1986), 247--263. 

\bibitem{Hattori2} H. Hattori, 
The existence and large time behavior of solutions to a system related to a phase transition problem,
SIAM J. Math. Anal. 34 (2003), 774--804.  

\bibitem{Hattori3} H. Hattori, 
Existence of solutions with moving phase boundaries in thermoelasticity,
J. Hyperbolic Differ. Equ. 5 (2008), 589--611
 
\bibitem{HayesLeFloch-scalar} B.T. Hayes and P.G. LeFloch,
Nonclassical shocks and kinetic relations. Scalar conservation laws,
Arch. Rational Mech. Anal. 139 (1997), 1--56.

\bibitem{HayesLeFloch-schemes} B.T. Hayes and P.G. LeFloch, 
Nonclassical shocks and kinetic relations. Finite difference schemes, 
SIAM J. Numer. Anal. 35 (1998), 2169--2194.

\bibitem{HayesLeFloch-systems} B.T. Hayes and P.G. LeFloch, 
Nonclassical shock waves and kinetic relations. Strictly hyperbolic systems, 
SIAM J. Math. Anal. 31 (2000), 941--991.   
(Preprint \# 357, CMAP, Ecole Polytechnique, Palaiseau, France, Nov. 1996.) 

\bibitem{HayesShearer} B.T. Hayes and M. Shearer,
Undercompressive shocks and Riemann problems for scalar conservation laws with non-convex fluxes,
Proc. Roy. Soc. Edinburgh Sect. A 129 (1999), 733--754. 

\bibitem{HKM} H. Holden, K.H. Karlsen, and D. Mitrovic, 
Zero diffusion-dispersion smoothing limits for a scalar conservation law with discontinuous flux function, 
International Journal of Differential Equations (2009). 

\bibitem{HouLeFloch} T.Y. Hou and P.G. LeFloch, 
Why nonconservative schemes converge to wrong solutions. Error analysis, 
Math. of Comput. 62 (1994), 497--530.

\bibitem{HRL} T.Y. Hou, P. Rosakis, and P.G. LeFloch, 
A level set approach to the computation of twinning and phase transition dynamics, 
J. Comput. Phys. 150 (1999), 302--331.

\bibitem{HowardZumbrun} P. Howard and K. Zumbrun,
Stability of undercompressive shock profiles,
J. Differential Equations 225 (2006), 308--360.  

\bibitem{HuLeFloch} J. Hu and P.G. LeFloch,  
$L\sp 1$ continuous dependence property for systems of conservation laws, 
Arch. Ration. Mech. Anal. 151 (2000), 45--93.  

\bibitem{Hwang1} S. Hwang, 
Kinetic decomposition for the generalized BBM-Burgers equations with dissipative term,
 Proc. Roy. Soc. Edinburgh Sect. A 134 (2004), 1149--1162. 

\bibitem{Hwang2} S. Hwang, 
Singular limit problem of the Camassa-Holm type equation,
 J. Differential Equations 235 (2007), 74--84. 

\bibitem{HwangTzavaras} S. Hwang and A. Tzavaras, 
Kinetic decomposition of approximate solutions to conservation laws: application to relaxation and diffusion-dispersion approximations,
Comm. Partial Differential Equations 27 (2002), 1229--1254.  

\bibitem{IL} T. Iguchi and P.G. LeFloch, 
Existence theory for hyperbolic systems of conservation laws with general flux-functions, 
Arch. Rational Mech. Anal. 168 (2003), 165--244.

\bibitem{Isaacson} 
E. Isaacson, D. Marchesin, C.F. Palmeira, and B.J. Plohr, 
A global formalism for nonlinear waves in conservation laws, 
Comm. Math. Phys. 146 (1992), 505--552. 

\bibitem{JMS} D. Jacobs, W.R. McKinney, and M. Shearer,
Traveling wave solutions of the modified Korteweg-deVries Burgers equation,
J. Differential Equations 116 (1995), 448--467.

\bibitem{Jerome} J.W. Jerome, 
The mathematical study and approximation of semi-conductor models. Large-scale matrix problems and the numerical solution of partial differential equations (Lancaster, 1992), pp.~157--204, 
Adv. Numer. Anal., III, Oxford Univ. Press, New York, 1994.  

\bibitem{JosephLeFloch1} K.T. Joseph and P.G. LeFloch, 
Singular limits in phase dynamics with physical viscosity and capillarity, 
Proc. Royal Soc. Edinburgh 137A (2007), 1287--1312.

\bibitem{JosephLeFloch2} K.T. Joseph and P.G. LeFloch,
Singular limits for the Riemann problem: general diffusion, relaxation, and boundary conditions,
C.R. Math. Acad. Sci. Paris 344 (2007), 59--64. 

\bibitem{KLR} 
F. Kissling, P.G. LeFloch, and C. Rohde, 
Singular limits and kinetic decomposition for a non-local diffusion-dispersion problem, 
J. Differential Equations 247 (2009), 3338--3356. 

\bibitem{KondoLeFloch} C. Kondo and P.G. LeFloch, 
Zero diffusion-dispersion limits for hyperbolic conservation laws, 
SIAM Math. Anal. 33 (2002), 1320--1329.

\bibitem{Kruzkov} S.N. Kruzkov,
First order quasilinear equations with several independent variables,
Mat. Sb. (N.S.) 81 (123) (1970), 228--25.

\bibitem{Kwon} Y.-S. Kwon, 
Diffusion-dispersion limits for multidimensional scalar conservation laws with source terms,
 J. Differential Equations 246 (2009), 1883--1893. 

\bibitem{LaforestLeFloch}  M. Laforest and P.G. LeFloch, 
Diminishing functionals for nonclassical entropy solutions selected by kinetic relations, 
Port. Math. (2010). 

\bibitem{LaforestLeFloch2}  M. Laforest and P.G. LeFloch, 
in preparation. 

\bibitem{Lax57} P.D. Lax, 
Hyperbolic systems of conservation laws, II, 
Comm. Pure Appl. Math. 10 (1957), 537--566. 

\bibitem{Lax} P.D. Lax, 
{\em Hyperbolic systems of conservation laws and the mathematical theory of
shock waves,}  
Regional Confer. Series in Appl. Math. 11, SIAM, Philadelphia, 1973. 

\bibitem{LL} P.D. Lax and C.D. Levermore,
The small dispersion limit of the Korteweg-deVries equation, 
Comm. Pure Appl. Math. 36 (1983), 253--290.

\bibitem{LeFloch-CPDE} P.G. LeFloch, 
Entropy weak solutions to nonlinear hyperbolic systems in nonconservative form, 
Comm. Part. Diff. Equa. 13 (1988), 669--727.

\bibitem{LeFloch-IMA} P.G. LeFloch, 
An existence and uniqueness result for two nonstrictly hyperbolic systems, 
in ``Nonlinear evolution equations that change type'', pp.~126--138, 
IMA Vol. Math. Appl., 27, Springer, New York, 1990. 

\bibitem{LeFloch-ARMA} P.G. LeFloch,
Propagating phase boundaries: formulation of the problem and existence via the Glimm scheme,
Arch. Rational Mech. Anal. 123 (1993), 153--197.

\bibitem{LeFloch-Freiburg} P.G. LeFloch, 
An introduction to nonclassical shocks of systems of conservation laws,  
International School on Hyperbolic Problems, Freiburg, Germany, Oct. 97, D. Kr\"oner, 
M. Ohlberger and C. Rohde eds., Lect. Notes Comput. Eng., Vol.~5, Springer Verlag, 1999, pp.~28--72.

\bibitem{LeFloch-book} P.G. LeFloch, 
{\em Hyperbolic systems of conservation laws: the theory of classical and
nonclassical shock waves,} 
Lecture Notes in Mathematics, E.T.H. Z\"urich, Birkh\"auser, 2002. 

\bibitem{LeFloch-graphs} P.G. LeFloch, 
Graph solutions of nonlinear hyperbolic systems, 
J. Hyperbolic Differ. Equ. 1 (2004), 643--689. 

\bibitem{LeFlochLiu} P.G. LeFloch and T.-P. Liu, 
Existence theory for nonlinear hyperbolic systems in nonconservative form, 
Forum Math. 5 (1993), 261--280.

\bibitem{LMR} P.G. LeFloch, J.-M. Mercier, and C. Rohde, 
Fully discrete entropy conservative schemes of arbitrary order, 
SIAM J. Numer. Anal. 40 (2002), 1968--1992.

\bibitem{LeFlochMishra} P.G. LeFloch and S. Mishra,  
in preparation. 

\bibitem{LeFlochMohamadian} P.G. LeFloch and M. Mohamadian, 
Why many shock wave theories are necessary.  Fourth-order models, kinetic functions, and equivalent equations, 
J. Comput. Phys. 227 (2008), 4162--4189. 

\bibitem{LeFlochNatalini} P.G. LeFloch and R. Natalini, 
Conservation laws with vanishing nonlinear diffusion and dispersion, 
Nonlinear Analysis 36 (1999), 213--230.

\bibitem{LeFlochRohde}  P.G. LeFloch and C. Rohde, 
High-order schemes, entropy inequalities, and nonclassical shocks, 
SIAM J. Numer. Anal. 37 (2000), 2023--2060.

\bibitem{LeFlochShearer} P.G. LeFloch and M. Shearer,
Nonclassical Riemann solvers with nucleation, 
Proc. Royal Soc. Edinburgh 134A (2004), 941--964.  

\bibitem{LeFlochThanh3} P.G. LeFloch and M.D. Thanh, 
Nonclassical Riemann solvers and kinetic relations. III. A nonconvex hyperbolic model for Van der Waals fluids,
Electr. J. Diff. Equa. 72 (2000), 1--19. 

\bibitem{LeFlochThanh1} P.G. LeFloch and M.D. Thanh, 
Nonclassical Riemann solvers and kinetic relations I. 
A nonconvex hyperbolic model of phase transitions, 
Z. Angew. Math. Phys. 52 (2001), 597--619. 

\bibitem{LeFlochThanh2}  P.G. LeFloch and M.D. Thanh, 
Nonclassical Riemann solvers and kinetic relations. II. An hyperbolic-elliptic model of phase transitions, 
Proc. Royal Soc. Edinburgh 132A (2001), 181--219. 

\bibitem{LeFlochThanh-Waals} P.G. LeFloch and M.D. Thanh, 
Properties of Rankine-Hugoniot curves for Van der Waals fluids, 
Japan J. Indust. Applied Math. 20 (2003), 211--238. 

\bibitem{LevyShearer1} R. Levy and M. Shearer,   
Comparison of two dynamic contact line models for driven thin liquid films, 
European J. Appl. Math. 15 (2004), 625--642.  
 
\bibitem{LevyShearer2} R. Levy and M. Shearer,  
Kinetics and nucleation for driven thin film flow,
 Phys. D 209 (2005), 145--163.
 
\bibitem{Li} H.-L. Li and R.-H. Pan, 
Zero relaxation limit for piecewise smooth solutions to a rate-type viscoelastic system in the presence of shocks,
J. Math. Anal. Appl. 252 (2000), 298--324.  

\bibitem{LPT} P.-L. Lions, B. Perthame, and E. Tadmor,  
A kinetic formulation of multidimensional scalar conservation laws and related equations, 
J. Amer. Math. Soc. 7 (1994), 169--191. 

\bibitem{Liu1} T.-P. Liu,  
The Riemann problem for general 2x2 conservation laws, 
Trans. Amer. Math. Soc. 199 (1974), 89--112. 
 
\bibitem{Liu2} T.-P. Liu, 
{\em Admissible solutions of hyperbolic conservation laws,} 
Mem. Amer. Math. Soc. 30, 1981.

\bibitem{LiuYang} T.-P. Liu and T. Yang,
Well-posedness theory for hyperbolic conservation laws,
Comm. Pure Appl. Math. 52 (1999), 1553--1586. 

\bibitem{LiuZumbrun1} T.-P. Liu and K. Zumbrun,
Nonlinear stability of an undercompressive shock for complex Burgers equation,
Comm. Math. Phys. 168 (1995), 163--186. 

\bibitem{LiuZumbrun2} T.-P. Liu and K. Zumbrun, 
Nonlinear stability of general undercompressive shock waves, 
Comm. Math. Phys. 174 (1995), 319--345.

\bibitem{MP1} J.-M. Mercier and B. Piccoli, 
Global continuous Riemann solver for nonlinear elasticity,
 Arch. Ration. Mech. Anal. 156 (2001), 89--119. 

\bibitem{MP2} J.-M. Mercier and B. Piccoli, 
Admissible Riemann solvers for genuinely nonlinear $p$-systems of mixed type,
J. Differential Equations 180 (2002), 395--426. 

\bibitem{MerkleRohde} C. Merkle and C. Rohde,
The sharp-interface approach for fluids with phase change: Riemann problems and ghost fluid techniques,
M2AN Math. Model. Numer. Anal. 41 (2007), 1089--1123. 

\bibitem{Munch} A. M\"unch,
Shock transitions in Marangoni gravity-driven thin-film flow,
Nonlinearity 13 (2000), 731--746. Ê	

\bibitem{NT} S.-C. Ngan and L. Truskinovsky L.,
Thermo-elastic aspects of dynamic nucleation, 
J. Mech. Phys. Solids 50 (2002), 1193--1229. 

\bibitem{Oleinik} O. Oleinik, 
Discontinuous solutions of nonlinear differential equations, 
Transl. Amer. Math. Soc. 26  (1963), 95--172.

\bibitem{OttoWestdickenberg} F. Otto and M. Westdickenberg, 
Convergence of thin film approximation for a scalar conservation law,
 J. Hyperbolic Differ. Equ. 2 (2005), 183--199.
 
\bibitem{Panov-systems} E.Y. Panov,
On the theory of entropy solutions of the Cauchy problem for a class of nonstrictly hyperbolic systems of conservation laws, 
 Sb. Math. 191 (2000), 121--150.

\bibitem{Panov-unique} E.Y. Panov,
On the theory of generalized entropy sub- and supersolutions of the Cauchy problem for a first-order quasilinear equation,  
Differ. Equ. 37 (2001), 272--280.

\bibitem{RaoofiZumbrun} M. Raoofi and K. Zumbrun, 
Stability of undercompressive viscous shock profiles of hyperbolic-parabolic systems,
J. Differential Equations 246 (2009), 1539--1567. 

\bibitem{Ratz} A. R\"atz and A. Voigt, 
A diffuse-interface approximation for surface diffusion including adatoms,
 Nonlinearity 20 (2007), 177--192. 
 
\bibitem{Rohde1}  
C. Rohde, Scalar conservation laws with mixed local and
  non-local diffusion-dispersion terms,
  SIAM J. Math. Anal. 37 (2005), 103--129.

\bibitem{Rohde2} C. Rohde, 
On local and non-local Navier-Stokes-Korteweg systems for liquid-vapour phase transitions,
 ZAMM Z. Angew. Math. Mech.  85 (2005),  839--857.  
 
\bibitem{SchecterShearer} S. Schecter and M. Shearer, 
Undercompressive shocks for nonstrictly hyperbolic conservation laws, 
Dynamics Diff. Equa. 3 (1991), 199--271. 

\bibitem{Schonbek} M.E. Schonbek, 
Convergence of solutions to nonlinear dispersive equations,
Comm. Part. Diff. Equa. 7 (1982), 959--1000. 

\bibitem{SchulzeShearer} S. Schulze and M. Shearer,
Undercompressive shocks for a system of hyperbolic conservation laws with cubic nonlinearity,
J. Math. Anal. Appl. 229 (1999), 344--362.

\bibitem{Shearer-86} M. Shearer, 
The Riemann problem for the planar motion of an elastic string,
J. Differential Equations 61 (1986), 149--163. 

\bibitem{Shearer1} M. Shearer, 
Dynamic phase transitions in a van der Waals gas,
 Quart. Appl. Math. 46 (1988), 631--636. 

\bibitem{ShearerYang} M. Shearer and Y. Yang, 
The Riemann problem for a system of conservation laws of mixed type with a cubic nonlinearity,
Proc. Roy. Soc. Edinburgh Sect. A 125 (1995), 675--699. 

\bibitem{Slemrod1} M. Slemrod, 
Admissibility criteria for propagating phase boundaries in a van der Waals fluid, 
Arch. Rational Mech. Anal. 81 (1983), 301--315.
 
\bibitem{Slemrod2} M. Slemrod, 
A limiting viscosity approach to the Riemann
problem for materials exhibiting change of phase,
Arch. Rational Mech. Anal. 105 (1989), 327--365.

\bibitem{Suliciu} I. Suliciu, 
On modelling phase transition by means of rate-type constitutive equations, shock wave structure, 
Int. J. Ing. Sci. 28 (1990), 827--841.

\bibitem{Tadmor-old} E. Tadmor,
The numerical viscosity of entropy stable schemes for systems of conservation laws. I,
Math. Comp. 49 (1987), 91--103. 

\bibitem{Tadmor-acta} E. Tadmor,
Entropy stability theory for difference approximations of nonlinear conservation laws and related time-dependent problems,
Acta Numer. 12 (2003), 451--512. 

\bibitem{Tadmor-recent} E. Tadmor and W. Zhong, 
Entropy stable approximations of Navier-Stokes equations with no artificial numerical viscosity,
J. Hyperbolic Differ. Equ. 3 (2006), 529--559. 
 
\bibitem{Truskinovsky1} L. Truskinovsky,  
Dynamics of non-equilibrium phase boundaries in a heat conducting 
nonlinear elastic medium, J. Appl. Math. and Mech. (PMM) 51 (1987), 777--784.  

\bibitem{Truskinovsky2} L. Truskinovsky,  
Kinks versus shocks, 
in ``Shock induced transitions and phase structures in general media'', 
R. Fosdick, E. Dunn, and M. Slemrod ed., IMA Vol. Math. Appl., Vol.~52, 
Springer-Verlag, New York (1993), pp.~185--229. 
           
\bibitem{Truskinovsky3} L. Truskinovsky, 
Transition to detonation in dynamic phase changes, 
Arch. Rational Mech. Anal. 125 (1994), 375--397. 

\bibitem{TV} L. Truskinovsky and A. Vainchtein, 
Quasicontinuum models of dynamic phase transitions,
Contin. Mech. Thermodyn. 18 (2006), 1--21. 

\bibitem{DuijnPeletierPop}
C.J. Van Duijn,  L.A. Peletier, and I.S. Pop, 
A new class of entropy solutions of the Buckley-Leverett equation,
SIAM J. Math. Anal. 39 (2007), 507--536.

\bibitem{Volpert} A.I. Volpert, 
The space BV and quasi-linear equations, Mat. USSR Sb. 2 (1967), 225--267.

\bibitem{Wendroff} B. Wendroff 
The Riemann problem for materials with non-convex equations of state. I. 
Isentropic flow, J. Math. Anal. Appl. 38 (1972), 454--466. 

\bibitem{ZHL} X.-G. Zhong, T.Y. Hou, and P.G. LeFloch, 
Computational methods for propagating phase boundaries, 
J. Comput. Phys. 124 (1996), 192--216.


\end{thebibliography}
\end{document}